\documentclass{amsart}
\usepackage[T1]{fontenc}
\usepackage[utf8]{inputenc}

\usepackage[hidelinks]{hyperref}
\usepackage[marginratio={1:1, 1:1}]{geometry}

\usepackage{microtype}
\usepackage{soul}

\usepackage{amsmath,amssymb}
\usepackage{amsthm,mathrsfs}
\usepackage[all, cmtip]{xy}
\usepackage{mathtools}
\usepackage[shortlabels]{enumitem}

\usepackage{float}

\usepackage{tikz}

\newcommand{\cO}{\mathcal{O}}
\newcommand{\cU}{\mathcal{U}}
\newcommand{\cP}{\mathcal{P}}
\newcommand{\p}{\mathbb{P}}

\newcommand{\Z}{\mathbb{Z}}

\DeclareMathOperator{\fix}{Fix}

\DeclareMathOperator{\Sym}{Sym}

\DeclareMathOperator{\supp}{supp}
\DeclareMathOperator{\codim}{codim}
\DeclareMathOperator{\Pic}{Pic}

\DeclareMathOperator{\prym}{Prym}
\DeclareMathOperator{\Fix}{Fix}
\DeclareMathOperator{\nm}{Nm}

\DeclareMathOperator{\sing}{sing}
\DeclareMathOperator{\aut}{Aut}
\DeclareMathOperator{\spec}{Spec}
\DeclareMathOperator{\reg}{reg}
\DeclareMathOperator{\id}{id}
\DeclareMathOperator{\gr}{gr}
\DeclareMathOperator{\NS}{NS}

\DeclareMathOperator{\PL}{PL}

\newcommand{\sheafext}{\mathscr{E}xt}

\newcommand{\sheafhom}{\mathscr{H}om}

\newtheorem{theorem}{Theorem}[section]

\newtheorem{corollary}[theorem]{Corollary}
\newtheorem{proposition}[theorem]{Proposition}
\newtheorem{propdef}[theorem]{Proposition-Definition}
\newtheorem{lemma}[theorem]{Lemma}
\newtheorem*{proposition*}{Proposition}
\newtheorem*{lemma*}{Lemma}

\newtheorem{remark}[theorem]{Remark}
\newtheorem{definition}[theorem]{Definition}
\newtheorem{example}[theorem]{Example}

\date{\today}
\title{Irreducible symplectic varieties via relative Prym varieties}
\author[E. Brakkee]{Emma Brakkee}
\address{\parbox{0.9\textwidth}{Emma Brakkee, Leiden University, Mathematical Institute,
Einsteinweg 55, 2333 CC Leiden, The Netherlands
\vspace{1mm}}}
\email{{e.l.brakkee@math.leidenuniv.nl}}
\author[C. Camere]{Chiara Camere}
\address{\parbox{0.9\textwidth}{Chiara Camere, Dipartimento di Matematica F. Enriques, Università degli Studi di Milano, Dipartimento di Matematica, Via
Cesare Saldini 50, 20133 Milano, Italy \vspace{1mm}}}
\email{{chiara.camere@unimi.it}}
\author[A. Grossi]{Annalisa Grossi}
\address{\parbox{0.9\textwidth}{Annalisa Grossi, Université Paris-Saclay, CNRS, Laboratoire de Mathématiques d’Orsay, Rue Michel Magat, Bât. 307, 91405 Orsay, France \vspace{1mm}}}
\email{{annalisa.grossi@universite-paris-saclay.fr}}
\author[L. Pertusi]{Laura Pertusi}
\address{\parbox{0.9\textwidth}{Laura Pertusi, Dipartimento di Matematica F. Enriques, Università degli studi di Milano, Via Cesare Saldini
50, 20133 Milano, Italy \vspace{1mm}}}
\email{laura.pertusi@unimi.it}
\author[G. Sacc\`{a}]{Giulia Sacc\`{a}}
\address{\parbox{0.9\textwidth}{Giulia Saccà, Columbia University, Department of Mathematics
2990 Broadway, New York, NY 10027, USA \vspace{1mm}}}
\email{{gs3032@columbia.edu}}
\author[S. Viktorova]{Sasha Viktorova}
\address{\parbox{0.9\textwidth}{Sasha Viktorova, Department of Mathematics, KU Leuven, Celestijnenlaan 200B, 3001 Leuven, Belgium \vspace{1mm}}}
\email{{sasha.viktorova@kuleuven.be}}

\begin{document}

\begin{abstract} 
Generalizing work of Markushevich--Tikhomirov and Arbarello--Saccà--Ferretti, we use relative Prym varieties to construct Lagrangian fibered symplectic varieties in infinitely many dimensions. We then give criteria for when the construction yields primitive symplectic varieties, respectively, irreducible symplectic varieties. 
The starting point of the construction is a K3 surface endowed with an anti-symplectic involution and an effective linear system on the quotient surface. We give sufficient conditions on the linear system to ensure that the relative Prym varieties satisfy the criteria above.
As a consequence, we produce infinite series of irreducible symplectic varieties.
\end{abstract}

\maketitle

\section{Introduction}

\subsection{Irreducible symplectic varieties}
Since the 1980's the Beauville--Bogomolov decomposition theorem has turned the spotlight on irreducible holomorphic symplectic manifolds, establishing them as one of the building blocks of compact K\"ahler manifolds with trivial first Chern class. More recently, the effort of many has resulted in the formulation of a decomposition theorem for singular varieties with trivial canonical class (see \cite{GKKP,DG,GrGuKe,Druel2018,Guenancia2016,Campana2021, BGL2020}\cite[Theorem 1.5]{HP}).
The singular version of the decomposition theorem states that a normal projective variety with numerically trivial canonical class and klt singularities admits a finite quasi-étale cover which is isomorphic to the product of abelian varieties, strict Calabi–Yau varieties, and irreducible symplectic varieties. These 
are the singular analogue of irreducible holomorphic symplectic manifolds. 

While irreducible holomorphic symplectic manifolds have provided an interesting sample of manifolds on which to test general conjectures, they are notoriously difficult to construct. 
Beyond K3 surfaces, which are the only examples in dimension $2$, the  known deformation classes in higher dimensions are those of Hilbert schemes of points on a K3 surface and of generalized Kummer manifolds, which occur in all even dimensions greater or equal to $4$, and those of O'Grady's examples which occur in dimension \(6\) and \(10\). 

Relaxing the smoothness assumption, more examples arise. However, it remains hard to find new examples and a classification still seems to be out of reach.
The known examples of irreducible symplectic varieties are terminalizations of symplectic quotients of irreducible holomorphic symplectic manifolds \cite{Fujiki, Fu-Menet, GM22, bertini2024terminalizations}, moduli spaces of semistable sheaves on K3 or abelian surfaces  or in the Kuznetsov component of a cubic fourfold or a Gushel--Mukai fourfold \cite{PR,Sacca-ISV}, or compactifications of Lagrangian fibrations, see \cite{MT,SS} (for \cite{MT} combine \cite[Proposition~3.12]{MenetRiess} and \cite[Proposition 3(2)]{Perego}).
Compactifications of Lagrangian fibrations were also studied in \cite{ASF} and \cite{Matteini}; we will see that these are  irreducible symplectic varieties as well (see Corollaries~\ref{cor_ASFisISV} and \ref{cor_Matteini_generalized}).

The purpose of this paper is to construct infinitely many examples of irreducible symplectic varieties in infinitely many dimensions. We do this by proving a general result on relative Prym varieties associated to a K3 surface with an anti-symplectic involution and invariant linear system on the K3 surface.

\subsection{Examples via relative Prym varieties} 
The construction of relative Prym varieties was proposed in papers by Markushevich--Tikhomirov \cite{MT} and Arbarello--Saccà--Ferretti \cite{ASF}. 
The idea is to consider a K3 surface $S$ carrying an anti-symplectic involution $i$ and, for any polarization $H$ and smooth invariant curve $D$ on $S$, the moduli space $M_H:=M_{S,H}(v)$ of $H$-stable sheaves on $S$ with Mukai vector $v=(0,D,1-g(D))$. 
One studies a certain irreducible component of the fixed locus inside $M_H$ of the symplectic birational involution $\tau$ obtained as the composition of $i^*$ and $F\mapsto \sheafhom(F,\cO_{\supp(F)})$,
which commute and are both anti-symplectic.
Over a smooth invariant curve $D'\in |D|$, this amounts to studying the Prym variety associated to the double cover $D'\rightarrow D'/i $.
Globally, one obtains a subvariety $\cP_H$ inside $M_H$ which has an open subset that is fibered in Prym varieties.
When the compactification $\cP_H$ is a symplectic variety, as we will show to be the case when $H=D$, it has a structure of Lagrangian fibration $\mathcal{P}_D\rightarrow |C|$. 
As mentioned above, the general fibers of this morphism are Prym varieties, so it will be called the relative Prym variety associated to $(S,i)$ and $|C|$.
Here $|C|$ is the linear system on the quotient surface that pullbacks to $|D|$ (one of the two when the double cover is étale).

The relative Prym variety associated to $(S,i)$ and $|C|$ has been investigated before in the case when $S$ is a double cover of a del Pezzo surface $T$ of degree 2 and $C=-K_T$ \cite{MT}, and when $S$ covers an Enriques surface $T$ and  $D$ is primitive  \cite{ASF}. A few more cases, when the quotient surface $T$ is a del Pezzo surface and $C=-K_T$ or $-2K_T$, have been studied in \cite{Matteini} and in \cite{SS} -- see Section \ref{sec:examples} for further details. 
In all of these cases, it was shown that one of the two following possibilities occurs: the relative Prym variety is singular and admits no smooth symplectic resolution, or a symplectic resolution exists and is an irreducible holomorphic symplectic manifold of $K3^{[n]}$-type.

The aim of this paper is to discuss the relative Prym construction for any choice of very general $(S,i)$ and $|C|$.
By work of Nikulin~\cite{Nikulin_Finitegroups} there are 75 different families of K3 surfaces carrying an anti-symplectic involution. Here, by very general $(S,i)$ we mean that $(S,i)$ is a very general point in the corresponding family. We can choose any effective linear system $|C|$ on the quotient surface $T$, giving many more cases than have been studied to date.
We will discuss criteria for the relative Prym varieties to be respectively symplectic, primitive symplectic and irreducible symplectic varieties (see Section \ref{sec:singular-definitions} for exact definitions).

After recalling the necessary definitions and constructions in Section \ref{sec:preliminaries}, we obtain the following result, which we prove in Section \ref{sec_sv} using arguments inspired by \cite{MT, ASF}. 

\begin{proposition}\label{prop_simplvarBeauville}
Let $S$ be a  smooth K3 surface with an anti-symplectic involution $i$ and let $f\colon S\rightarrow T=S/i$ be the quotient map, let $C$ be a smooth curve of genus $g(C)$ on $T$ and let $D=f^{-1}C$. 
We assume that the pair $(S,i)$ is very general in the sense of Definition \ref{def_generalK3withsympinv} and that $D$ is smooth of  genus  $g(D)\geq 2$. 
Then the relative Prym variety $\mathcal{P}_D \to |C|$ is a symplectic variety of dimension $2(g(D)-g(C))$.
\end{proposition}

In particular, $\mathcal{P}_D$ has trivial canonical bundle and canonical singularities. In Section \ref{sec_sv} we  establish  general criteria for a projective symplectic variety to be primitive symplectic (Proposition \ref{prop_condition_for_psv}) and irreducible symplectic (Proposition \ref{prop_keyprop}). The first application is to note that the (normalization of the) relative Prym varieties constructed in \cite{ASF} are irreducible symplectic varieties. Then, we show the following result.

\begin{theorem}\label{thm_psv}
Under the assumptions of Proposition \ref{prop_simplvarBeauville}, if in addition $|C|$ is very ample on $T$ and $|D|$ is very ample on $S$, then $\mathcal{P}_D$ is a primitive symplectic variety.
\end{theorem}

The key ingredient in the proof of Theorem \ref{thm_psv} is the construction of a dominant rational map $S^{[g(D)-g(C)]}\dashrightarrow \mathcal{P}_D$, which generalizes the one constructed in \cite[Theorem 8.1]{ASF}. This is the content of Proposition \ref{prop_dominantMap}.
It is worth noting that even in other cases which may not satisfy the assumptions of Theorem \ref{thm_psv}, showing the existence of such a dominant rational map from an irreducible holomorphic symplectic variety  to $\mathcal{P}_D$ automatically implies that $\mathcal{P}_D$ is a primitive symplectic variety. This fact is proved Proposition \ref{prop_condition_for_psv} -- see also Remark \ref{rmk_rationalmap_Cample}.

In order to obtain an irreducible symplectic variety, Proposition \ref{prop_keyprop} requires showing that the regular locus $(\mathcal{P}_D)_{\reg}\subset\mathcal{P}_D$ is simply connected. This is the most difficult part of the proof, and the one which requires more assumptions. In fact, in Section \ref{sec:pi1} we generalize the strategy of \cite{MT,ASF} and in Theorem~\ref{thm_pi1} we give a criterion for the relative Prym variety to be simply connected. We show the following statement.

\begin{theorem}\label{thm_isv_exactHp}
Under the assumptions of Proposition \ref{prop_simplvarBeauville}, the relative Prym variety $\mathcal{P}_D$ is an irreducible symplectic variety of dimension $2(g(D)-g(C))$ if the following hold:

\begin{enumerate}
\item $|C|$ and $|D|$ are very ample on $T$ and $S$, respectively;
\item the locus $\{\Gamma\in |C|\mid f^*\Gamma\ \mathrm{is}\ \mathrm{not}\ \mathrm{integral}\}$ has codimension at least $2$ in $|C|$;
\item Let $Z\subset|C|$ be a codimension $1$ irreducible component of the locus $\{\Gamma\in |C|\mid f^*\Gamma\ \text{is singular}\}$.
The general element of $Z$ is one of the following:
 \begin{enumerate}[(a)]
    \item[(1)] A smooth integral curve intersecting $B$ transversely except at one point, where the multiplicity is 2;
    \item[(2)] An integral curve intersecting the branch locus $B\subset T$ of $f$ transversely, with one node outside $B$ and no other singularities.
 \end{enumerate}
\end{enumerate}
\end{theorem}
 
In Section \ref{section_curvesreduciblepreim}, respectively Section \ref{section_curvessingularpreim}, we investigate which properties of $|C|$ imply  that the locus of curves in $|C|$ with non-integral preimage has codimension $\geq 2$, respectively that the codimension 1 components of $|C|$ are as in Theorem \ref{thm_isv_exactHp}(3). This is the subject of Theorem \ref{thm_curvesreduciblepreim}, respectively of Theorem \ref{thm_codim1loci}. The final outcome is the following criterion, which is the main result of the paper.

\begin{theorem} \label{thm_Pisv}
Under the assumptions of Proposition \ref{prop_simplvarBeauville}, the relative Prym variety $\mathcal{P}_D$ is an irreducible symplectic variety of dimension $2(g(D)-g(C))$ if the following hold:

\begin{enumerate}\label{thm_isv}
\item \label{hyp_veryample} $|C|$ and $|D|$ are very ample on $T$ and $S$, respectively;
\item \label{hyp_int>2} $C.B>2$;
\item \label{hyp_int_not4} $C^2\neq 4$ or $C.B\neq 4$;
\item \label{hyp_2conn} $|C|$ is $2$-connected;
\item \label{hyp_hyperell} $C$ is not hyperelliptic if $B^2\leq 0$.
\end{enumerate}
\end{theorem}

Theorem \ref{thm_isv} allows us to show in infinitely many cases that the relative Prym variety is an irreducible symplectic variety.
In Section \ref{sec:examples} we show that the linear systems $|\cO_{\mathbb{P}^2}(n)|$, $n\geq 3$, on $\mathbb{P}^2$ satisfy all the assumptions of Theorem~\ref{thm_isv}, thus producing examples of irreducible symplectic varieties of infinitely many dimensions, starting from dimension 18.
We also extend the examples of \cite{MT} and \cite{Matteini} to del Pezzo surfaces of higher degree, and we give examples of linear series on del Pezzo surfaces that were not studied in the past literature and which satisfy the assumptions of Theorem \ref{thm_isv}. This yields examples of irreducible symplectic varieties of arbitrarily high dimension,
starting from dimension 8.
Note that the conditions (1)--(5) are numerical, and we expect them to be satisfied for sufficiently positive linear systems $|C|$ on $T$. 

Note also that assumptions (2)--(5) above are sufficient to deduce the simple connectedness of $(\mathcal{P}_D)_{\reg}$, but they are not necessary. For example, they do not hold in the case studied by Markushevich--Tikhomirov even though the relative Prym variety still turns out to be an irreducible symplectic variety (see \cite[Proposition 3.12]{MenetRiess} and \cite[Proposition 3(2)]{Perego}).
When the linear system does not satisfy the assumptions, for example if there are  codimension one components of the locus $\{\Gamma\in |C|\mid f^*\Gamma\ \text{is singular}\}$ that do not satisfy the  assumptions of Theorem \ref{thm_isv_exactHp}, we still expect that our general strategy for checking if the corresponding relative Prym is an irreducible symplectic variety may work, provided one proves an analogue of Proposition \ref{prop_fundgroupV} to compute the extra generators of the fundamental group that come from loops around the additional divisorial components.

Although we leave to a future work the computation of the Euler characteristic and of other topological invariants of the relative Prym varieties that we construct in this paper, we should mention that we expect  our construction to produce genuinely new examples. This is indeed the case for some of the low-dimension relative Prym varieties that have been studied before:
in the case of Markushevich--Tikhomirov's example, the deformation class of $\cP_D$ is distinct from that of a moduli space of semistable sheaves on K3 surfaces (those studied by \cite{PR}). Indeed, Markushevich--Tikhomirov's example is singular and does not admit a symplectic resolution, while in dimension four those moduli spaces are smooth or have a symplectic resolution. 
Matteini's six-dimensional example $\cP_D$ is also not deformation equivalent to a six dimensional moduli space of stable sheaves on a K3 surface: it is birational to the quotient of a Hilbert scheme of three points on a K3 surface by a symplectic involution, so its second Betti number is strictly smaller than $23$, while the second Betti number of any moduli space of semistable sheaves on a K3 surface, whether smooth or singular, is equal to $23$ (see \cite{PR}).

Finally, note that our Theorems \ref{thm_psv}, \ref{thm_isv_exactHp}, \ref{thm_isv} do not apply in the case when $D$ is hyperelliptic.
In this situation, we expect $\cP_D$ to be birational to a moduli space of stable sheaves on a K3 surface. This is indeed what happens for hyperelliptic linear systems in the case of Enriques surfaces \cite[Section 6]{ASF} (see Remark \ref{remark_hyperell}), but we have not pursued this at this time.

\subsection*{Acknowledgements}
This project started during the first edition of "Women in Algebraic Geometry" hosted online by ICERM in July 2020 and the authors want to heartily thank all the organizers and ICERM for the occasion. Moreover, we want to thank Nikolas Adaloglou, Enrico Arbarello, Arend Bayer, Lie Fu, Alice Garbagnati, Marco Golla, Christian Lehn, Mirko Mauri, Arvid Perego, Chris Peters, for useful and interesting discussions. We would also like to thank Stefan Kebekus for asking whether the varieties constructed in \cite{ASF} are examples of irreducible symplectic varieties and, more generally, how one can find examples of such varieties. This question contributed to the start of this project. 

E.B. was supported by NWO grants 016.Vidi.189.015 and VI.Veni.212.209.
C.C. received partial support from grants Prin Project 2020 "Curves, Ricci flat Varieties and their Interactions" and Prin Project 2022 "Symplectic varieties: their interplay with Fano manifolds and derived categories"; she is a member of the Indam group GNSAGA.
A.G. was supported by
the European Research Council (ERC) under the European Union’s Horizon 2020 research
and innovation programme (ERC-2020-SyG-854361-HyperK), and by the DFG through the research grant Le 3093/3-2.
L.P. is a member of the Indam group GNSAGA.
G.S. was partially supported by NSF CAREER grant DMS-2144483 and NSF FRG grant DMS-2052750.
S.V. was supported by Methusalem grant METH/21/03 -- long term structural funding of the Flemish Government.

\section{Preliminaries}\label{sec:preliminaries}

In this section we collect all the results and the definitions that we need in the rest of the paper. 
First, we briefly recall Nikulin's classification of anti-symplectic involutions acting on K3 surfaces, 
and then we discuss in detail the construction of the relative Prym variety, focusing in particular on the case of a quotient rational surface.

\subsection{Anti-symplectic involutions on K3 surfaces}\label{subsec_involutions}
Here we summarize some results of Nikulin on K3 surfaces with an anti-symplectic involution, contained in \cite{Nikulin_Finitegroups}, \cite{Nikulin_integral} and \cite{Nikulin_quotientgroups.geomappl}. 
Consider a pair \((S, i)\) consisting of a K3 surface $S$ with an anti-symplectic involution $i$, that is, $i^*\omega_S=-\omega_S$ for any symplectic form $\omega_S$ on $S$. By \cite[Corollary~3.2]{Nikulin_Finitegroups}, such a K3 surface \(S\) is projective.
For $(S,i)$ very general in an irreducible component of the moduli space of K3 surfaces with anti-symplectic involution, 
the Néron--Severi group $\NS(S)$, which always contains the fixed lattice $H^2(S,\mathbb{Z})^{i^*}:=\{x\in H^2(S,\mathbb{Z})\mid i^*x=x\}$, is as small as possible.

\begin{definition}\label{def_generalK3withsympinv}
    A K3 surface $S$ with an anti-symplectic involution $i$ is called \emph{very general} if $\NS(S)=H^2(S,\mathbb{Z})^{i^*}$. 
\end{definition}

From now on, we assume that $(S,i)$ is very general. Let \(T=S/i\) be the quotient surface, which is smooth. 
When $\fix(i)$ is empty, $T$ is an Enriques surface; when $\fix(i)$ is non-empty, $T$ is a (not necessarily minimal) rational surface.

The classification of pairs $(S,i)$ is obtained using lattice theory. We denote by \(L_{\pm} \subset H^{2}(S,\mathbb{Z})\) the sublattices of cohomology classes \(l\) such that \(i^{*}(l)=\pm l\). 
Let \(r\) be the rank of \(L_+\), then the signature of \(L_+\) is \((1,r-1)\). 
Denote by \(D_{L_+}=L_{+}^{\vee}/L_{+}\) the discriminant group of $L_+$, with induced quadratic form \(q_+\colon D_{L_+} \to \mathbb{Q}/2\mathbb{Z} \). 
Then we have  \(D_{L_+} \cong (\mathbb{Z}/2\mathbb{Z})^{a}\) for some non-negative integer $a$, that is, $L_+$ is \emph{2-elementary}. 
The parity \(\delta\in\{0,1\}\) of \(q_+\) is defined to be 0 if \(q_+(D_{L_+}) \subset \mathbb{Z}/2\Z\), and 1 otherwise. 
The triple \((r,a,\delta)\) is called the \emph{main invariant} of \((S,i)\): by \cite[Theorem~3.6.2]{Nikulin_integral}, the isometry classes of the lattices \(L_{\pm}\) and the isomorphism class of the pair $(S,i)$ are uniquely determined by \((r,a,\delta)\). 
Moreover, the following holds: 

\begin{theorem}[{\cite[Theorem~4.2.2]{Nikulin_quotientgroups.geomappl}}]\label{Nikulinstheorem}
Let \((S,i)\) be a very general K3 surface with an anti-symplectic involution with main invariant \((r,a,\delta)\).
\begin{itemize}
    \item[i)] If \((r,a,\delta)=(10,10,0)\), then \(\Fix(i)=\emptyset\).
    \item[ii)] If \((r,a,\delta)=(10,8,0)\), then \(\Fix(i)\) is the union of two elliptic curves.
    \item[iii)] Otherwise \(\Fix(i)\) decomposes as \(C \sqcup E_1 \sqcup \cdots \sqcup E_k \) where \(C\) is a genus \(g\) curve and \(E_1, \cdots, E_k\) are rational curves with
    \begin{equation*}
        g=11-\frac{r+a}{2}, \ \  k=\frac{r-a}{2}.
    \end{equation*}
     Moreover, \(\delta=0\) if and only if the class of \(\Fix(i)\) is divisible by 2 in \(L_+\).
\end{itemize}
\end{theorem}
Using \cite[\S1.12]{Nikulin_integral}, one can classify all even 2-elementary lattices of signature $(1,r)$ that admit a primitive embedding into the second integral cohomology of a K3 surface.
Any of these lattices can be realized as fixed sublattice for some anti-symplectic involution $i$ on a K3 surface $S$ (see \cite[\S4.2]{Nikulin_quotientgroups.geomappl}).
The 75 possible main invariants $(r,a,\delta)$ were first given in \cite[Table~1]{Nikulin_quotientgroups.geomappl}. The following picture was taken from \cite[Figure 3.1]{Matteini}:

\begin{figure}[H]
\label{NikulinsTriangle}
 \includegraphics[width=11cm]{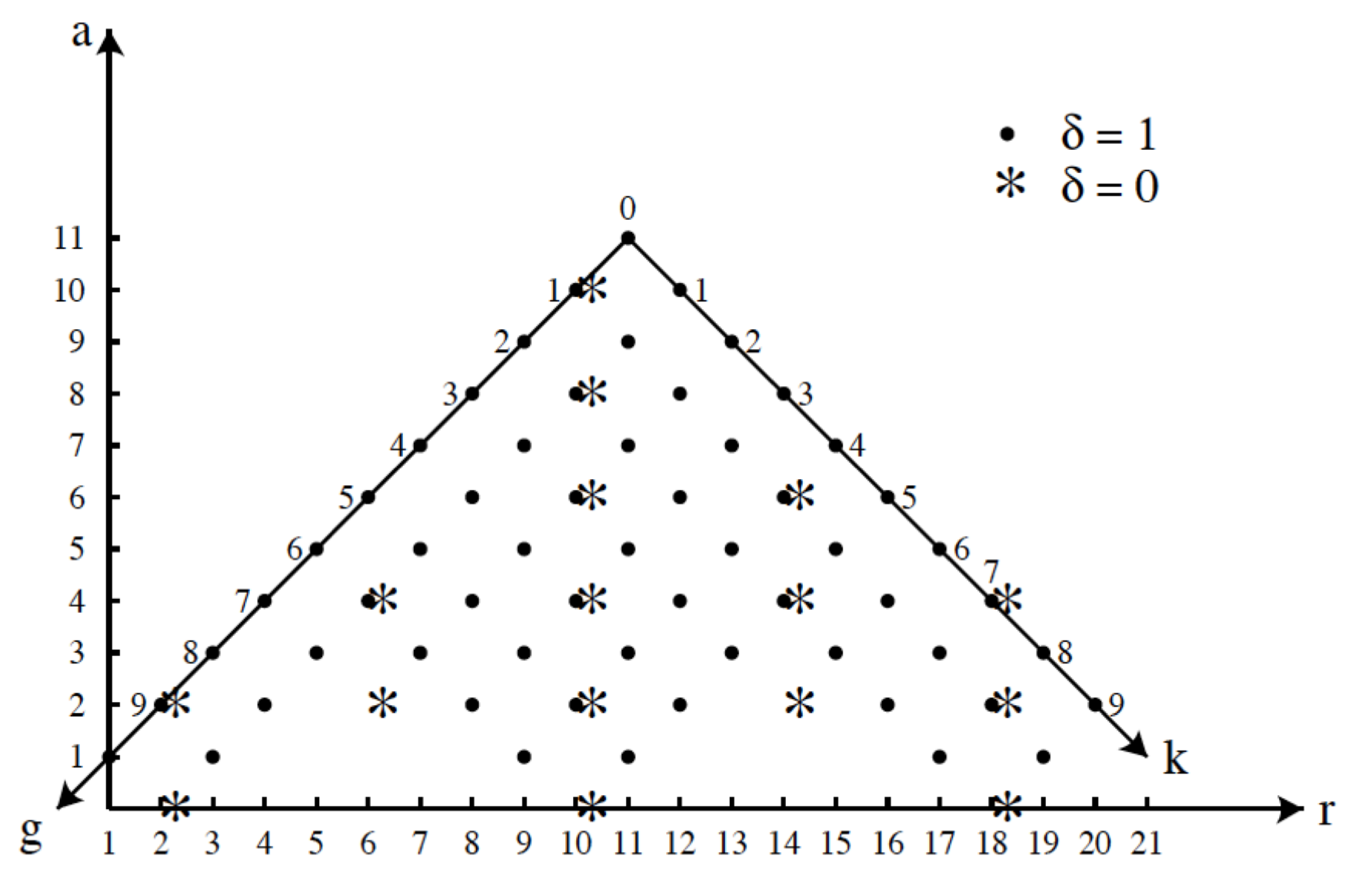}
\caption{Possible main invariants for $\NS(S)$}
\end{figure}

\subsection{Construction of the relative Prym variety} \label{subsec-rel-prym}

The construction of the relative Prym variety, first described by Markushevich and Tikhomirov \cite{MT} for a K3 cover of a del Pezzo surface of degree 2, has been developed in full generality by \cite{Sacca,ASF, Matteini}.
We recall the construction here. Since the case of a  double cover of an Enriques surface was studied in detail in \cite{ASF}, we focus on the case of double covers of rational surfaces. This corresponds precisely to the case when the double cover is not \'etale (and hence ramifies along a smooth curve).

\subsubsection{The relative compactified Jacobian}
\label{subsubsec-rel-jac}

Before starting the construction of the relative Prym variety, we need to recall the construction and some properties of Beauville--Mukai's integrable systems.

We fix a K3 surface $S$ and a smooth curve $D\subset S$ of genus $g(D)\geq 2$. By \cite{Saint-Donat} (cf. also \cite[Corollary 2.3.6]{Huy_bookK3}), the linear system $|D|$ has no base points and its general member is a smooth irreducible curve.
Consider the following Mukai vector:
\[v:=(0,[D],1-g(D))\in H^0(S,\mathbb{Z})\oplus H^2(S,\mathbb{Z})\oplus H^4(S,\mathbb{Z}).\]
For an ample line bundle $H$ on $S$, we consider the Beauville--Mukai system corresponding to $v$ and $H$, i.e. the moduli space 
\[
M_H:=M_{S,H}(v)
\]
of Gieseker $H$-semistable coherent sheaves on $S$ of Mukai vector $v$. 
Thus $M_H$ parametrizes S-equivalence classes of $H$-semistable sheaves of pure dimension one with first Chern class $[D]$ and Euler characteristic $1-g(D)$.
The space $M_H$ is an irreducible projective variety of dimension  $2g(D)$.
Since we do not require $H$ to be $v$-generic, nor the Mukai vector $v$ to be primitive,  $M_H$ can be singular.
Its regular locus $(M_{H})_{\reg}$, which contains the locus $M_H^{st}$ of stable sheaves, 
carries a symplectic form \cite{Mukai}.

The determinantal, or Fitting, support defines a morphism (\cite{LePoitier})
\begin{equation}\label{eq_supportmap}
\supp\colon M_H\to |D|,\;\; F\mapsto\supp(F).
\end{equation}
We will see in Proposition \ref{prop-M-symplectic-variety} that $M_H$ is a (possibly singular) symplectic variety (see Definition \ref{def_symplvariety}). 
When $M_H$ is smooth, so it has a symplectic form everywhere, it is well-known that the support morphism is a Lagrangian fibration.
We will see in Section~\ref{sec:singular-definitions} that this is still the case when $M_H$ is singular.

\medskip 
If $\supp(F)$ is an integral curve $D_0$, then $F$ is stable \cite[Lemma~1.1.13]{Sacca}, hence $[F]\in (M_H)_{\reg}$. 
In this case, $F$ is of the form $g_*L$ with $L$ a torsion-free sheaf of rank 1 on $D_0$ and $g$ the embedding of $D_0$ into $S$. 
In particular, for a smooth curve $D_0\in |D|$, the fibre $\supp^{-1}(D_0)$ is the Jacobian $J(D_0)$ of degree 0 line bundles on $D_0$. 
More generally, if $D_0$ is integral, its fibre is the compactified Jacobian $\overline{J(D_0)}$ parametrizing torsion free sheaves on $D_0$ of rank 1 and degree 0.
Furthermore, over the locus parametrizing reduced curves, the support morphism realizes this moduli space as a relative compactified Jacobian. The fiber of $\supp$ over a point corresponding to a non-reduced curve is more complicated, but, somewhat inappropriately, we will still refer to $M_H$ as the relative compactified Jacobian of the linear system $|D|$.

The map $\supp\colon M_H\to |D|$ has a rational 0-section 
\begin{equation}\label{zerosection_supportmap}
s\colon |D|\dashrightarrow M_H,
\end{equation}
which is regular on the locus $V'\subset |D|$ of integral curves. It sends $D_0\in V'$ to  $\iota_*\mathcal{O}_{D_0}$, where $\iota$ is the embedding $D_0\hookrightarrow S$.
This is a stable sheaf, hence $s(V')$ is contained in $(M_H)_{\reg}$.

Note that $v$ is primitive if and only if $[D]\in H^2(S,\mathbb{Z})$ is primitive. 
In this case, if we further assume that $H$ is $v$-generic, i.e.\ $H$-semistable sheaves with Mukai vector $v$ are actually stable, then $M_H$ is a smooth projective irreducible holomorphic symplectic manifold \cite{Mukai,Y}.
In general, when $v$ is not primitive or $H$ is not $v$-generic, $M_H$ is still a symplectic variety, as we record in Proposition \ref{prop-M-symplectic-variety} below.

In terms of the topology of $M_H$, we have the following general result on its fundamental group.

\begin{proposition} \label{prop_Msimplyconn}
   For any Mukai vector $w=(0,[C],s)\in H^*(S,\mathbb{Z})$ and any polarization $H$, the moduli space $M_{S,H}(w)$ is simply connected. 
\end{proposition}
\begin{proof}
    When $H$ is $w$-generic, this is 
    proved in \cite[\S 3.1]{PR}. More precisely, the authors prove that in this case the locus $M_{S,H}(w)^{irr}$ of sheaves supported on irreducible curves is simply connected. 
    Now let $H'$ be an arbitrary polarization. Since $M_{S,H}(w)^{irr}=M_{S,H'}(w)^{irr}$, the result follows since $M_H$ is normal by \cite[Theorem 5.1]{AS-update}, and the fundamental group of a normal variety is always surjected upon by the fundamental group of an open subset \cite[0.7.B]{FuLaz}.
\end{proof}

We finish by recalling a natural stratification of the singular locus $(M_H)_{\sing}$ of $M_H$ by locally closed smooth symplectic varieties,
as explained in \cite[Prop.~2.5]{AS}.
Let $[F]\in (M_H)_{\sing}$ be a strictly $H$-semistable sheaf. 
Then there are positive integers $a_1,\dots,a_r$ and $n_1,\dots,n_r$, and Mukai vectors $v_1,\dots,v_r$, such that $F$ is S-equivalent 
to a unique polystable sheaf
\[F_{1,1}^{n_1}\oplus\cdots\oplus F_{1,a_1}^{n_1}\oplus F_{2,1}^{n_2}\oplus\cdots\oplus F_{2,a_2}^{n_2}\oplus \cdots\oplus F_{r,1}^{n_r}\oplus\cdots\oplus F_{r,a_r}^{n_r}\]
where the $F_{j,i}$ are pairwise distinct stable sheaves with Mukai vector $v_j$, so that $\sum_{j=1}^ra_jn_jv_j=v$, and the pairs $(n_j,v_j)$ are mutually distinct. 
We say that $F$ is \emph{of type $(a_1n_1v_1,\dots,a_rn_rv_r)$}. 
Let $\Sigma_{(a_1n_1v_1,\dots,a_rn_rv_r)}$ be the locus of all polystable sheaves of type $(a_1n_1v_1,\dots,a_rn_rv_r)$, and denote by $M_{H,v_i}$ the moduli space $M_{S,H}(v_i)$. Then the rational map 
\begin{align*}
\Sym^{a_1}M_{H,v_1}\times\cdots\times\Sym^{a_r}M_{H,v_r} &\dashrightarrow \Sigma_{(a_1n_1v_1,\dots,a_rn_rv_r)},\\ 
((F_{1,1},\dots,F_{1,a_1}),\dots,(F_{r,1},\dots,F_{r,a_r})) &\longmapsto F
\end{align*}
is birational, and its restriction to the open subset  
\(\left(\Sym^{a_1}M_{H,v_1}^{st}\right)_{\reg}\times\cdots\times\left(\Sym^{a_r}M_{H,v_r}^{st}\right)_{\reg}\)
is bijective.
Since $M_{H,v_j}^{st}$ is contained in $(M_{H,v_j})_{\reg}$, it carries a symplectic form; therefore, 
$\Sigma_{(a_1n_1v_1,\dots,a_rn_rv_r)}$ (or rather its smooth locus) 
has a symplectic form as well.
The union of all strata $\Sigma_{(a_1n_1v_1,\dots,a_sn_sv_s)}$ is $(M_H)_{\sing}$, and only finitely many of them are non-empty, giving the stratification we aimed for.

\begin{remark}
Since $M_H$ is a symplectic variety (see Proposition~\ref{prop-M-symplectic-variety}), the existence of a stratification of $(M_H)_{\sing}$ also follows abstractly from \cite[\S3]{Kaledin}.
\end{remark}

\subsubsection{The relative Prym variety}
\label{subsubsec_Prymfib}
The starting point of the construction of the relative Prym variety is to consider the moduli space $M_H$ studied in Section~\ref{subsubsec-rel-jac}, in the case of a K3 surface $S$ endowed with an anti-symplectic involution $i$ and an $i$-invariant linear system on $S$. We refer to \S 3.3 of \cite{ASF} for a more thorough introduction (note that in loc.\ cit.\ this is formulated for \'etale double covers, but most of the results of that introductory section work in our setting too, mutatis mutandi).

\medskip
From now on, we fix a very general pair $(S,i)$ as in Definition~\ref{def_generalK3withsympinv}. As in Section \ref{subsec_involutions}, we denote by $T$ the smooth quotient surface $S/ i$,  by $f\colon S\to T$ the quotient morphism, and by $B\subset T$ the branch locus of $f$. The branch locus satisfies $B\in|-2K_T|$.

Let $C\subset T$ be a smooth curve with $C^2>0$ and let $D:=f^{-1}C\subset S$ be its preimage under $f$. In particular, $D$ is $i$-invariant.
We assume that $D$ is smooth, which happens when $C$ intersects the branch locus $B$ transversely (see Lemma~\ref{lemma_milnornumber}).
The genus $g(D)$ of $D$ is
\[g(D)=2g(C)-C.K_T-1.\]

Let $C_0\in|C|$ be any smooth curve whose inverse image $D_0:=f^{-1}C_0$ is smooth as well. 
Then $f$ restricts to a double cover $D_0\to C_0$ branched in $B.C=-2C.K_T$ many points, and this cover is induced by $\omega_T^{-1}|_C$. 
Note that it cannot be a trivial cover: if $D_0$ is a disjoint union of curves $D_0',D_0''\subset S$ isomorphic to $C_0$, then $(D_0')^2,(D_0'')^2>0$ which contradicts the Hodge index theorem.
It follows that $\omega_T|_C\not\cong\cO_C$.
The map $D_0\to C_0$ has covering involution $i_0:=i|_{D_0}$.
We recall the definition of the associated Prym variety \cite[\S3]{Mumford}: 

\begin{definition}\label{def_prym_smooth_curves}
Let $-i_0^*$ be the involution on $J(D_0)$ defined by
\[
L \mapsto i_0^*L^{\vee}.
\]
The Prym variety $\prym(D_0/C_0)$ associated to the double cover $D_0\rightarrow C_0$ is  is the identity component of the fixed locus:
\[\prym(D_0/C_0):=\fix^0(-i_0^*)=\ker(\id +i_0^*)^0\subset J(D_0).\]
\end{definition}

Alternatively, one can consider the norm map
\begin{align*}
 \nm\colon J(D_0)&\to J(C_0)\\
 L=\mathcal{O}_{D_0}\left(\sum a_ix_i\right)&\mapsto \det(f_*L)=\mathcal{O}_{C_0}\left(\sum a_if(x_i)\right);
\end{align*}
since $f^*(\nm(L))=(\id+i_0^*)L$, we have
\[\prym(D_0/C_0)=\ker(f^*\circ\nm)^0.\]

The Prym variety is an abelian variety of dimension $g(C)-C.K_T-1=g(D)-g(C)$, which is principally polarized if and only if $C.B=0$ or $C.B=2$ \cite[Corollary~2]{Mumford}.
If $i_0$ is fixed-point free, then $\fix(-i_0^*)$ has four connected components \cite[\S3.3]{ASF}. If $i_0$ is ramified, it has only one:
    \begin{proposition}\label{fix_curves_connected}
Let $f_0\colon D_0\to C_0$ be a ramified double cover of smooth curves, with covering involution $i_0$. Then $\fix(-i_0^*)\subset J(D_0)$ is connected. In particular, $\prym(D_0/C_0)=\fix(-i_0^*)$. 
\end{proposition}
\begin{proof}
    By \cite[Section~3]{Mumford}, the pullback $f^*_0\colon J(C_0)\to J(D_0)$ is injective. It follows that $\fix(-i_0^*)=\ker(f^*_0\circ\nm)$ equals $\ker(\nm)$, which is connected by \cite[Lemma~1.1]{Kanev}. 
\end{proof}

Definition \ref{def_prym_smooth_curves} can be extended to the case when $C_0$ and $D_0$ are integral and locally planar.

\begin{propdef} \label{involution_integral_case}
Let $f_0\colon D_0\to C_0$ be a double cover of integral, locally planar curves, with covering involution $i_0$. The assignment
\[
L \mapsto i_0^* \sheafhom(L, \mathcal{O}_{D_0})
\]
defines an involution $\tau_0$ on the compactified Jacobian of the curve $D_0$.
We define the compactified Prym variety of $D_0 \to C_0$ by setting
\[
\overline{ \prym(D_0/C_0)}:=\Fix^0(\tau_0),
\]
where $\Fix^0(\tau_0)$ denotes the irreducible component containing the identity of the fixed locus of the involution $\tau_0$.
\end{propdef}
\begin{proof}
    The dual of a rank one torsion free sheaf of degree $0$ on a locally planar curve is still rank one, torsion free, of degree $0$. Moreover, it can be checked that this assignment works in families, therefore defining an automorphism on the degree $0$ compactified Jacobian of $D_0$. Since it is an involution on the locus parametrizing line bundles, it is an involution on the whole compactified Jacobian.
\end{proof}

In the situation of Proposition-Definition~\ref{involution_integral_case}, certain properties of usual Prym varieties still hold.
We denote by $g_a(C)$ the arithmetic genus of a curve $C$.

\begin{lemma} \label{integral Prym}
    Let $f_0\colon D_0 \to C_0$ be a double cover of integral locally planar curves. Then the compactified Prym variety $\overline{\prym(D_0/C_0)}$ has dimension $g_a(D_0)-g_a(C_0)$, i.e. the same dimension as in the smooth case.
\end{lemma}
\begin{proof}
    To compute $\dim\overline{\prym(D_0/C_0)}$, it is enough to compute the dimension of its tangent space at $\mathcal{O}_{D_0}$. The involution $\tau_0$ fixes $\mathcal{O}_{D_0}$ and the $+1$ eigenspace for the action of $\tau_0$ on $T_{\mathcal{O}_{D_0}}\Pic^0(D_0)$ is the tangent space to $\prym(D_0/C_0)$ at $\mathcal{O}_{D_0}$. Moreover, the $-1$ eigenspace is the tangent space at $\mathcal{O}_{D_0}$ to $\Fix(i_0^*)$. It is therefore enough to compute the dimension of $\Fix(i_0^*)$ at $\mathcal{O}_{D_0}$.
    
    Let $L$ be an $i_0^*$-invariant line bundle on $D_0$ that belongs to the irreducible component $\Fix^0(i_0^*)$ of $\Fix(i_0^*)$ containing $\mathcal{O}_{D_0}$. Since $i_0^*$ acts as the identity on the fibres of $L$ over the ramification points of $f_0$, 
    the line bundle $L$ descends to $C_0$, that is $L\cong f_0^*L'$ for a line bundle $L'$ (see e.g. \cite[Theorem 2.3]{DrNar}). 
    One can set $L'$ to be the locally free sheaf $((f_0)_*L)^{i_0^*}$ of $i_0^*$-invariant sections of $(f_0)_*L$.
    Therefore, $\Fix^0(i_0^*)=f_0^*\Pic^0(C_0)$ and $\dim\prym(D_0/C_0)=\dim\Pic^0(D_0)-\dim\Fix^0(i_0^*)=\dim\Pic^0(D_0)-\dim\Pic^0(C_0)=g_a(D_0)-g_a(C_0)$.
\end{proof}

The following two lemmas will be important in Section~\ref{sec:pi1}:

\begin{lemma} \label{irr_prym_one_node}
    Let $f_0\colon D_0 \to C_0$ be a double cover such that $C_0$ is smooth and $D_0$ is integral with one node. Assume that the induced involution on the normalization of $D_0$ is not \'etale. Then $\Fix^0(\tau_0)$ is the only irreducible component of $\Fix(\tau_0)$ of maximal dimension.
\end{lemma}
\begin{proof}
    Let $Z$ be an irreducible component of $\Fix(\tau_0)$ different from $\Fix^0(\tau_0)$. Since the induced involution $i_0'$ on the normalization of $D_0$ is not \'etale, it follows from Proposition \ref{fix_curves_connected} that $Z$ has to be contained in the boundary of $\overline{J(D_0)}\setminus \Pic^0(D_0)$ of the compactified Jacobian $\overline{J(D_0)}$. We thus need to understand the action of $\tau_0$ on the boundary.
  
    Let $\nu\colon D_0' \to D_0$ be the normalization map. The boundary of $\overline{J(D_0)}$ is naturally isomorphic to $\Pic^{-1}(D_0')$ via the pushforward map $\nu_*$, which is a closed embedding (see e.g.\ \cite[Lemma 3.1]{B}).  
    Let $i_0'$ be the covering involution of the induced map $D_0'\to C_0$.
    Let $p$ and $q$ be the inverse images of the node under $\nu$; in particular, $i_0'(p)=q$.
    Define an involution $\tau_0'$ on $\Pic^{-1}(D_0')$ by setting $L' \mapsto (i_0')^*(L')^{\vee}(-p-q)$.

    By relative duality applied to the normalization map $\nu$, and the fact that the relative dualizing sheaf of $\nu$ is $\cO_{D'_0}(-p-q)$, it follows that $\nu_*\colon \Pic^{-1}(D_0') \to \overline{J(D_0)}$ is equivariant with respect to the involution $\tau_0$ on $\overline{J(D_0)}$ and the involution $\tau_0'$ on $\Pic^{-1}(D_0')$:
    \[\tau_0(\nu_*L') = i_0^*\sheafhom(\nu_*L',\cO_{D_0})\cong i_0^*\nu_*\sheafhom(L',\cO_{D_0'}(-p-q))\cong \nu_*((i_0')^*(L')^{\vee}(-p-q))=\nu_*\tau_0'(L').\]
    Thus to understand $Z$ it suffices to understand the fixed locus of $\tau_0'$ on $\Pic^{-1}(D_0')$.
    For this aim, consider the isomorphism $\phi\colon \Pic^0(D_0') \to \Pic^{-1}(D_0')$ given by tensoring with $\cO_{D_0'}(-p)$. 
    Since $i(p)=q$, one can check that $\phi$ is equivariant with respect to the involutions $\tau_0'$ on $\Pic^{-1}(D_0')$ and $(i_0')^*(\cdot)^{\vee}$ on $\Pic^{0}(D_0')$. Hence the fixed locus of $\tau_0'$ on $\Pic^{-1}(D_0')$ is isomorphic to the fixed locus of $(i_0')^*(\cdot)^{\vee}$ on $\Pic^{0}(D_0')$. But this is just the usual Prym variety of the double cover $D'_0 \to C_0$, and hence it has dimension $g(D'_0)-g(C_0)=g_a(D_0)-g(C_0)-1$. This proves the lemma.
\end{proof}

\begin{lemma} \label{irr_prym_transv_B}
    Let $f_0\colon D_0 \to C_0$ be a double cover of integral curves such that $C_0$ has one node and $D_0$ has two nodes. If $f_0$ is not \'etale, then $\Fix^0(\tau_0)$ is the only irreducible component of $\Fix(\tau_0)$ of maximal dimension.
\end{lemma}
\begin{proof}
    Let $\nu\colon D_0' \to D_0$ be the normalization map. Similarly to the previous lemma, the set of sheaves in $\overline{J(D_0)}$ which are not locally free at both nodes of $D_0$ can be identified with $\Pic^{-2}(D_0')$.
    Denote by $i_0'$ the covering involution of $f_0'\colon D_0' \to C_0'$, where $f_0'$ is the lift of $f_0\circ\nu$ to the normalization $C_0'$ of $C_0$. 
    We define an involution $\tau_0'$ on $\Pic^{-2}(D_0')$ as $L' \mapsto (i_0')^*(L')^{\vee}(-p_1-q_1-p_2-q_2)$ where $p_j$ and $q_j$ are the preimages of the nodal points of $D_0$ in $D_0'$. As in the proof of Lemma~\ref{irr_prym_one_node}, $\nu_*\circ\tau_0'$ is compatible with $\tau_0$.

    Let $Z$ be an irreducible component of $\Fix(\tau_0)$ different from $\Fix^0(\tau_0)$. Since $i_0'$ is not \'etale, $Z$ has to be contained in the set of non-locally free sheaves in $\overline{J(D_0)}$. As in the proof of Lemma~\ref{irr_prym_one_node}, it follows that 
    \[\dim Z\leq \dim\Fix(\tau_0')=\dim\prym(D_0',C_0')=g_a(D_0')-g_a(C_0')=g_a(D_0)-g_a(C_0)-1. \qedhere\]
\end{proof}

\begin{remark}
In fact, in the situations of Lemmas~\ref{irr_prym_one_node} and \ref{irr_prym_transv_B}, one can show that $\fix(\tau_0)$ is irreducible.
\end{remark}

Following \cite{MT,ASF}, we globalize the construction of $\prym(D_0/C_0)$ over all of $|C|$ as follows. We fix the Mukai vector $v=(0,[D],1-g(D))\in H^*(S,\mathbb{Z})$ and we consider, for a polarization $H$ on $S$, the moduli space  $M_H$ discussed in Section \ref{subsubsec-rel-jac}.
The pullback of sheaves induces a birational involution
\[i^*\colon M_H\dashrightarrow M_H,\]
regular on the locus of sheaves with integral support and compatible with the support morphism.
By \cite[Lemma~3.6]{ASF}, $i^*$ is anti-symplectic (see also \cite[Lemma~3.4]{OW}).
There is a second anti-symplectic birational involution on $M_H$ \cite[Proposition~3.11]{ASF}, given by
\begin{align*}
 j\colon M_H&\dashrightarrow M_H\\
 F&\mapsto \sheafext^1_S(F,\mathcal{O}_S(-D)).
\end{align*}
If $F\in M_H$ is supported on $D_0$, 
then $j(F)\cong\sheafhom_{D_0}(F,\mathcal{O}_{D_0})$ \cite[Lemma~3.7]{ASF}, which is again supported on $D_0$.
Since $\mathcal{O}_S(-D)$ is $i^*$-invariant, the maps $j$ and $i^*$ commute; hence the composition
\begin{equation}\label{eq_deftau}
\tau:=j\circ i^*\colon M_H\dashrightarrow M_H
\end{equation}
is a symplectic birational involution. 
We denote by $\fix(\tau)$
the fixed locus of the restriction of $\tau$ to the biggest open in $M_H$ where $\tau$ is regular. 
Then the intersection $\fix(\tau)\cap (M_H)_{\reg}=\fix(\tau|_{(M_H)_{\reg})}$, which is a Zariski dense open subset of $\fix(\tau)$, is smooth and has a symplectic form \cite[Proposition~2.6]{Fujiki}. 
Note that the support map restricts to a morphism
\[\fix(\tau)\to |C|\]
where $|C|$ sits in $|D|$ via the pullback $f^*$.
The general fibre of $\fix(\tau)\to |C|$ is Lagrangian. 

\medskip
If $C_0\in |C|$ is a smooth curve such that $D_0=f^{-1}C_0$ is also smooth, then the fibre of $\fix(\tau)\to |C|$ over $C_0$ is $\fix(-i|_{D_0}^*)$.
In particular, it contains the image $[\cO_{D_0}]$ of $C_0$ 
under the rational section $s\colon |C|\dashrightarrow M_H$.
We denote by $\fix^0(\tau)$ the irreducible component of $\fix(\tau)$ containing the image of $s$. 
Then $\fix^0(\tau)$ dominates $|C|$; the inverse image of the smooth curve $C_0\in|C|$ under $\fix^0(\tau)\to|C|$ is $\prym(D_0/C_0)$. Note that $\fix^0(\tau)$ has dimension
\[\dim|C|+\dim\prym(D_0/C_0)=\dim|C|+g(D)-g(C).
\]
We find the dimension of $|C|$ using the following lemma.

\begin{lemma}
\label{lemma_dimlinearsystemT}
 Let $C'$ be a curve on a rational surface $T$. We have $h^1(C',\cO_T(C')|_{C'})=h^0(C',\omega_T|_{C'})$
 and 
 \[\dim |C'|=h^0(C',\cO_T(C')|_{C'})=(C')^2+1-g(C')+h^0(C',\omega_T|_{C'}).\]
\end{lemma}
\begin{proof}
Note that $C'$ is Gorenstein, so the 
dualizing sheaf $\omega_{C'}$ is locally free, and is isomorphic to $(\cO_T(C')\otimes\omega_T)|_{C'}$ by adjunction. 
Serre duality gives $h^0(C',\omega_T|_{C'})=h^1(C',\cO_T(C')|_{C'})$.
Because the plurigenera of $T$ are all trivial, we have a short exact sequence
\[0\to H^0(T,\cO_T)\to H^0(T,\cO_T(C'))\to H^0(C',\cO_T(C')|_{C'})\to 0.\]
This gives $\dim|C'|=h^0(C',\cO_T(C')|_{C'})$; by the Riemann--Roch theorem for embedded curves, this equals $(C')^2+1-g(C')+h^0(C',\omega_T|_{C'})$.
\end{proof}

\begin{corollary}\label{cor_dimCforB.C>0}
If $C'$ is a curve on a rational surface $T$ with $\deg\omega_T|_{C'}\leq 0$ and $\omega_T|_{C'}\not\cong\cO_{C'}$,
then we have
$H^1(C',\cO_T(C')|_{C'})=0$ and 
\[\dim|C'|=(C')^2+1-g(C') = g(C')-C'.K_T-1.\] 
\end{corollary}
\begin{proof}
If $\deg\omega_T|_{C'}=K_T.C'$ is non-positive and $\omega_T|_{C'}\not\cong\cO_{C'}$, then we have $h^0(C',\omega_T|_{C'})=0$.
\end{proof}

Since the curve $C$ satisfies the assumptions of Corollary~\ref{cor_dimCforB.C>0}, we find

\begin{corollary}
The variety $\fix^0(\tau)$ has dimension
\[\dim\fix^0(\tau) = 2(g(D)-g(C)).
\]
\end{corollary}

The stratification of $(M_H)_{\sing}$ described in Section~\ref{subsubsec-rel-jac} induces stratifications by smooth symplectic varieties of $\fix(\tau)\cap (M_H)_{\sing}$ and $\fix^0(\tau)\cap (M_H)_{\sing}$.
Indeed, suppose that a polystable sheaf\[F= F_{1,1}^{n_1}\oplus\cdots\oplus F_{1,a_1}^{n_1}\oplus \cdots\oplus F_{r,1}^{n_r}\oplus\cdots\oplus F_{r,a_r}^{n_r}\]
in $\Sigma_{(a_1n_1v_1,\dots,a_rn_rv_r)}$
is fixed by $\tau$. In particular, $\tau(F_{h,k}):=j(i^*F_{h,k})$
is well-defined for all $h,k$.
One checks that $\tau(F_{h,k})$ has the same Mukai vector as $F_{h,k}$.
Consider the birational, bijective morphism
\begin{equation}\label{eq_stratum}
\left(\Sym^{a_1}M_{H,v_1}^{st}\right)_{\reg}\times\cdots\times\left(\Sym^{a_r}M_{H,v_r}^{st}\right)_{\reg} \to \Sigma_{(a_1n_1v_1,\dots,a_rn_rv_r)}.
\end{equation}
The map $F_{h,k}\mapsto j(i^*F_{h,k})$ is a well-defined involution on the $a_k$-fold product $M^{st}_{H,v_k}\times\cdots\times M^{st}_{H,v_k}$ for each $k$,
inducing a component-wise birational symplectic involution $\tau'$ on the left hand side of 
\eqref{eq_stratum}.
Its restriction to $\Sigma_{(a_1n_1v_1,\dots,a_rn_rv_r)}$ is the same as the restriction of $\tau$ on $M_H$ to $\Sigma_{(a_1n_1v_1,\dots,a_rn_rv_r)}$. 
Hence, $\fix(\tau)\cap \Sigma_{(a_1n_1v_1,\dots,a_rn_rv_r)}$ is an open subset of $\fix(\tau')$, which carries a symplectic form.

\medskip
When $B.C=0$, one can use this stratification to show that $\fix(\tau)$ has four irreducible components of dimension equal to $\dim\fix(\tau)=2(g(D)-g(C))$ which all dominate $|C|$ \cite[Lemma~3.2.10]{Sacca}.
When $B.C>0$, it follows from Proposition~\ref{fix_curves_connected} that $\fix^0(\tau)$ is the only irreducible component of $\fix(\tau)$ of maximal dimension, as we show in the following corollary.

\begin{corollary}\label{cor_Fix=connected}
If $B.C>0$, 
then the locus $\fix^0(\tau)$ is the unique irreducible component of $\fix(\tau)$ whose dimension equals $\dim\fix(\tau)=2(g(D)-g(C))$.
\end{corollary}
\begin{proof}
Let $Z$ be an irreducible component of $\fix(\tau)$. 
If $Z$ intersects $(M_H)_{\reg}$, then 
by Fujiki
the open subset $Z\cap (M_H)_{\reg}\subset Z$ has a symplectic form, 
with respect to which the general fibre of the restriction of the support map $\eta\colon Z\to\supp(Z)\subset |C|$ 
 is isotropic.
In particular, the dimension of a fibre is at most $(\dim Z)/2$. It follows that
\[\dim Z\leq(\dim Z)/2+\dim\eta(Z)\]
and therefore, $\dim Z\leq 2\dim\eta(Z)$.
Hence, if $\dim Z=\dim\fix^0(\tau)=2\dim|C|$, then $\eta$ is dominant.
By Proposition~\ref{fix_curves_connected},
for a general $C_0\in|C|$, the fibre of $\eta$ over $C_0$ equals the fibre over $C_0$ of $\fix^0(\tau)\to |C|$. We conclude that $Z=\fix^0(\tau)$.
 
Now let $Z$ be an irreducible component of $\fix(\tau)$ which is contained in $(M_H)_{\sing}$.
Then $Z\to |C|$ is not dominant.
As explained above, the stratification of $(M_H)_{\sing}$ induces a stratification of $Z$ into locally closed subsets with a symplectic form. 
In particular, the stratum of highest dimension induces a symplectic form on an open subset of $Z$, 
so that the general fibre of $\eta \colon Z\to\supp(Z)\subset|C|$ is isotropic.
As above, this implies that
\[\dim Z\leq 2\dim\eta(Z)<2\dim|C|=\dim\fix(\tau). \qedhere\]
\end{proof}

We compactify $\fix^0(\tau)$ by taking its closure $\overline{\fix^0(\tau)}$ in $M_H$. Finally, we define the relative Prym variety by normalizing $\overline{\fix^0(\tau)}$.

\begin{definition}\label{def_Prym}
Let $S$ be a very general K3 surface carrying an anti-symplectic involution $i$, and let $C$ be a smooth curve on the quotient surface $T$ with $C^2>0$, 
such that $D:=f^{-1}C$ is smooth.
Let $H$ be a polarization on $S$.
The \emph{relative Prym variety} $\mathcal{P}_H$ associated to these data is the normalization of the closure
$\overline{\fix^0(\tau)}$
of $\fix^0(\tau)$ in $M_H$.
\end{definition}

If \(H\) is \(i\)-invariant, then the birational involution $i^*$ of $M_H$ is biregular by \cite[Lemma~3.6]{ASF} (see also \cite[Proposition~3.1]{OW}). Furthermore, if \(H\) is a multiple of \(D\), then the anti-symplectic involution \(j\) defined on $M_H$ is biregular \cite[Lemma 3.2.8]{Matteini}. 
It follows that if $H=D$, then \( \tau =j \circ i^{*} \) is a biregular symplectic involution on \(M_D\). 
Hence $\fix(\tau)\subset M_D$ is closed, and, as a consequence,  the relative Prym variety \(\mathcal{P}_D\) is simply the normalization of
$\fix^0(\tau)$ in \(M_D\).

\begin{remark}
\label{remark_hyperell}
Note that we allow $i^*= j$ on $M_H$, in which case the relative Prym variety equals $M_H$. 
This happens when the linear system $|D|$ is hyperelliptic and $i$ induces the hyperelliptic involution on $D$. In fact, these conditions are necessary:
when $i^*=j$, the restriction $i_{|D}$ to an invariant smooth curve $D$ satisfies $(i_{|D})^*=-1$ on $J(D)$, hence $D$ is hyperelliptic by the Torelli theorem for curves. 
For an example,
take the main invariant of $(S,i)$ to be $(1,1,1)$. Then $S$ is a generic double cover of $\mathbb{P}^2$ branched along a smooth sextic $B$.
    When $C\in |\cO_{\mathbb{P}^2}(1)|$, the curve $D$ is hyperelliptic and $\varphi_D\colon S\rightarrow\mathbb{P}^2$ coincides with the quotient map $f$, hence $i^*=j$.

    It can happen, however, that $|D|$ is hyperelliptic but $i$ does not induce the hyperelliptic involution on $D$, so $i^*\neq j$. We thank Alice Garbagnati for suggesting the following example. Suppose that $i$ fixes a hyperelliptic curve $D$ of genus two, which holds when the main invariant of $i$ is $(r,18-r,\delta)\neq (10,8,0)$. Then the hyperelliptic involution induced by the degree two map $\varphi_D\colon S\rightarrow\mathbb{P}^2$ is different from $i$, since $i_{|D}=\id_D$. Therefore, we have $i^*\neq j$ on $M_H$. 

    Note that the case when $|D|$ is a hyperelliptic linear system is not covered by Theorems~\ref{thm_psv}--\ref{thm_isv}, since we need $D$ to be very ample to verify the criterion (Proposition \ref{prop_condition_for_psv}) for being an irreducible symplectic variety.
    However, similarly to the case of Enriques involutions studied in \cite{ASF},  we expect that in this case the relative Prym variety $\cP_H$ is birational to a moduli space of semistable sheaves on a K3 surface. 
\end{remark}

The variety $\mathcal{P}_H$ is a normal, compact, irreducible variety of dimension $2(g(D)-g(C))$, and we will see in Section \ref{sec:singular-definitions} that the smooth locus has a holomorphic $2$-form $\sigma$ which is symplectic on an open subset of $(\cP_H)_{\reg}$.
The support map induces a fibration
\[\eta\colon\mathcal{P}_H\to |C|\] 
whose general fibre is Lagrangian.
We will discuss when $\cP_H$ is a symplectic variety in the next section.

\begin{remark}\label{remark_sectionsupportP}
Recall that $\fix^0(\tau)$ is defined as the irreducible component of $\fix(\tau)$ containing the image of $|C|$ under the rational section \eqref{zerosection_supportmap} of the support map $\supp\colon M_H\to |D|$.
This rational section is defined on the locus $V'\subset|D|$ of integral curves and sends $V'$ into $(M_H)_{\reg}$. 
Hence, by \cite[Proposition 2.6]{Fujiki}, it sends the locus $U'\subset |C|\subset |D|$ of curves with integral preimage into $\fix^0(\tau)_{\reg}$.
Therefore, the induced rational section $s\colon |C|\dashrightarrow \cP_H$ of $\eta$ is defined on $U'$ as well, and sends $U'$ into $(\cP_H)_{\reg}$.
\end{remark}

\section{Symplectic varieties}\label{sec_sv}

In this section we first review some basic definitions and results on singular symplectic varieties. Then we show that the relative Prym varieties introduced in Section \ref{subsec-rel-prym} are symplectic varieties, at least when the involution $\tau$ (cf. Section  \ref{sec:singular-definitions}) is regular, and that they admit a Lagrangian fibration. We then prove 
a criterion for symplectic varieties to be irreducible symplectic (Proposition \ref{prop_keyprop}) which, together with Proposition \ref{prop_dominantMap} and the results of Section \ref{sec:pi1}, will allow us to prove that in many cases the relative Prym varieties are irreducible symplectic.

\subsection{Singular symplectic varieties}\label{sec:singular-definitions}

For a more thorough introduction to the content of this section and for examples we recommend \cite{Perego}, and references therein.

The first definition, which was introduced by Beauville \cite{Beauville2000}, is that of symplectic variety. 
Let $X$ be a normal variety. Denote by $X_{\reg}$ its regular locus, with its open embedding $\iota \colon X_{\reg} \to X$. For every integer $0 \leq p \leq \dim X$, set $\Omega^{[p]}_X:=\iota_*\Omega^p_{X_{\reg}} \cong \left(\bigwedge\nolimits^p\Omega_X\right)^{**}$. By definition, a reflexive $p$-form on $X$ is a global section of $\Omega^{[p]}_X$. 

\begin{definition}{(\cite[Definition 2.10]{Perego})}
\label{def_symplform}
Let $X$ be a normal variety. A \emph{symplectic form} on $X$ is a reflexive $2$-form $\sigma$, which is non-degenerate at every point of the regular locus $X_{\reg}$ of $X$.
\end{definition}

\begin{definition}{(\cite[ Definition 1.1]{Beauville2000}; cf also \cite[Definition 2.10]{Perego})} \label{def_symplvariety}
A variety $X$ is called a \emph{symplectic variety} if:
\begin{enumerate}[(i),itemsep=0pt]
    \item $X$ is normal;
    \item $X$ has a symplectic form $\sigma$;
    \item\label{condition_res_sing} 
    for every resolution of singularities $\rho\colon\widetilde{X}\to X$, the holomorphic symplectic form $\sigma_{\reg}:=\sigma|_{X_{\reg}}$ 
    extends to a holomorphic 2-form on $\widetilde{X}$.
\end{enumerate}
\end{definition} 

Note that it is enough to check condition \ref{condition_res_sing} for one resolution, since any two resolutions are dominated by a common one.
A resolution $\rho\colon \widetilde{X}\to X$ as in \ref{condition_res_sing} is called a \emph{symplectic resolution} if $\sigma_{\reg}$ extends to a holomorphic \emph{symplectic} form on $\widetilde{X}$.

Note also, that a symplectic variety has trivial canonical bundle and canonical singularities. In fact, it has rational Gorenstein singularities \cite[Prop. 1.3]{Beauville2000}.

\medskip
Now we introduce primitive symplectic varieties (see \cite[Definition 3.1]{BaLe} for the definition in the setting of K\"ahler, non-projective spaces).

\begin{definition}{(\cite[Definition 1(3)]{Schwald})}\label{def_psv}
 A \emph{primitive symplectic variety} is a normal projective variety $X$ such that $H^1(X, \cO_X)=0$ and $H^0(X,\Omega^{[2]}_X)$ is generated by a symplectic form $\sigma$.
\end{definition}

\noindent The above notion is the most general which allows for a deformation and moduli theory as in the smooth case, see \cite{BaLe}.

Finally, we give the definition of irreducible symplectic variety.
Recall that if $X$ and $Y$ are normal irreducible projective varieties, then a \emph{finite quasi-\'etale morphism} $g \colon X \to Y$ is a finite morphism which is \'etale in codimension one (see \cite[Definition 2.3]{grebharmonic}). The pullback morphism of forms on the smooth locus induces a morphism of reflexive sheaves $g^*\Omega_X^{[p]} \to \Omega_Y^{[p]}$, and thus a morphism $g^{[*]} \colon H^0(X, \Omega_X^{[p]}) \to H^0(Y, \Omega_Y^{[p]})$ on reflexive forms called reflexive pullback.

\begin{definition}[{\cite[ Definition 2.11]{Perego}}]
\label{def_irreduciblesymplectic}
A symplectic variety $X$ with symplectic form $\sigma$ is an \emph{irreducible symplectic variety} if $X$ is projective and for every finite quasi-\'etale morphism $g\colon Y\to X$, the exterior algebra of reflexive forms on $Y$ is spanned by $g^{[*]}\sigma$. 
\end{definition}

As an example, we show that the moduli spaces $M_H$ considered in this paper (see \S \ref{subsubsec-rel-jac}) are symplectic varieties.

\begin{proposition}\label{prop-M-symplectic-variety}
The moduli space $M_ H$ is a symplectic variety. 
\end{proposition}
\begin{proof}
This is well-known when $H$ is $v$-generic: if $v$ is primitive, then $M_H$ is  an irreducible holomorphic symplectic manifold; if $v$ is not primitive, then $M_H$ is either a symplectic variety which admits a symplectic resolution of O'Grady 10 type (\cite{LehnSorger}), or it is an irreducible symplectic variety (see \cite{PR} and references therein). 

If $H$ is not $v$-generic, we check the conditions in Definition \ref{def_symplvariety}.
By assumption the linear system $|D|$ contains an integral curve $D_0$, thus $M_H^{st}$ is non-empty, since the pushforward of $\cO_{D_0}$ to $S$ is stable with Mukai vector $v$. Moreover, by \cite[Theorem 5.1]{AS-update}, $M_H$ is normal. Let $\sigma$ be the symplectic form on $M_H$.

Let $H'$ be a $v$-generic polarization that lies in a $v$-chamber adjacent to the wall that $H$ lies in. Then by \cite[Lemma~4.1]{Zowislok} (see also \cite[Proposition 2.5]{AS}), there exists a birational morphism $\phi\colon M_{H'}\rightarrow M_H$ sending $F$ to $\gr_{H}F$, the direct sum of the graded pieces of its Jordan--Hölder filtration with respect to $H$, 
which is an isomorphism over $M_H^{st}$.
It follows that $\phi^*\sigma$ is symplectic on $M_{H'}$,
and where it is defined, it is equal to the symplectic form $\sigma'$ on $(M_{H'})_{\reg}$.
If $v$ is primitive, $\phi$ is in fact a symplectic resolution of singularities of $M_H$. If $v$ is not primitive, consider any resolution $\rho \colon \widetilde{M_{H'}} \to M_{H'}$: since $M_{H'}$ is a symplectic variety, $\rho^*\sigma'=\rho^*\phi^*\sigma$ extends to a holomorphic 2-form on $\widetilde{M_{H'}}$.
As $\phi\circ \rho\colon\widetilde{M_{H'}}\to M_H$ is a resolution of $M_H$, this completes the proof.
\end{proof}

As it is done in \cite{B} and \cite{Sawon} for smooth moduli spaces $M_H$, one shows that the support morphism $\supp$ defined in \eqref{eq_supportmap} is a Lagrangian fibration on $M_H$, in the sense of the following definition:

\begin{definition}[{\cite[Definition 1.2]{Matsushita-higher-direct}}]\label{def_lagfib}
Let $X$ be a symplectic variety with symplectic form $\sigma$. 
Let $B$ be a normal variety. A proper surjective morphism $g: X\rightarrow B$  with connected fibers is said to be a \emph{Lagrangian fibration} if a general fibre $F$ of $g$ is a Lagrangian subvariety of $X$ with respect to $\sigma$. That is, $\dim F =\frac{1}{2}\dim X$ and the restriction of $\sigma$ to $F\cap X_{\reg}$ identically vanishes.
\end{definition}

In fact, Matsushita has shown that this is the only possible kind of fibration when the source is an irreducible holomorphic symplectic manifold, and his theorem has been generalized to primitive symplectic varieties in \cite{Kamenova-Lehn}.

\medskip
We will now see that the singular locus of a symplectic variety $M$ has a generically non-degenerate holomorphic 2-form, and this 2-form is preserved by any symplectic automorphism of $M$.
We would like to thank Ch.~Lehn for directing us to Step 2 of \cite{Kaledin}, which we use in the following proposition.

\begin{proposition}[{\cite{Kaledin}}] \label{prop_kaledin}
    Let $M$ be a symplectic variety with singular locus $\Sigma$. The symplectic form on $M$ induces a holomorphic 2-form on an open subset of $\Sigma$ which is generically non-degenerate. 
    More precisely, let $\pi\colon \widetilde{M} \to M$ be    
    a projective resolution and let $\sigma_{\widetilde{M}}$ be the holomorphic 2-form on $\widetilde M$ extending the symplectic form on $M$. 
    There exists a dense open subset $U\subset\Sigma$ and a symplectic form $\sigma_U$ on $U$ such that, denoting $V:=\pi^{-1}(U)_{\reg}$, we have ${(\pi^*\sigma_U)}_{|V}={\sigma_{\widetilde{M}}}_{|V}$.
    Moreover, the symplectic form $\sigma_U$ is uniquely defined.
\end{proposition}

\begin{proof}
    The existence of the holomorphic 2-form on $V$ follows from the fact that the fibers of $\widetilde M \to M$ are rationally connected by \cite{HaconMcKernan} (see also Lemma 2.9 of \cite{Kaledin}). The fact that this 2-form is non-degenerate follows from Step 2 on pg. 145 of \cite{Kaledin}. Uniqueness follows from the definition of $\sigma_U$.
\end{proof}

\begin{remark} \label{rmk_kaledin}
    Note that the proof of Proposition~\ref{prop_kaledin} applies to any stratum of the singular locus of $M$.
    We will use this in the proof of Proposition~\ref{prop:fixed-locus-symplectic-variety} below. 
    Note also that Kaledin's result is stronger than what is stated above, since he shows that the singular locus of $M$ is stratified in smooth strata with holomorphic symplectic forms. However, for Proposition~\ref{prop:fixed-locus-symplectic-variety} we do not need the full result, but only that the holomorphic 2-form is generically non-degenerate.
\end{remark}

\begin{corollary} \label{cor_kaledin}
    Let $M$ be a symplectic variety with singular locus $\Sigma$. Suppose $\iota\colon M \to M$ is a symplectic automorphism 
    and let $\iota_\Sigma\colon \Sigma \to \Sigma$ be the induced automorphism on $\Sigma$. Then $\iota_\Sigma$ preserves the holomorphic $2$-form of Proposition \ref{prop_kaledin}.
\end{corollary}
\begin{proof}
    This follows from the uniqueness of the holomorphic form in Proposition \ref{prop_kaledin}.
\end{proof}

Next, we show that the normalization of the fixed locus of a symplectic finite order automorphism of a symplectic variety is itself a symplectic variety. 

\begin{proposition}\label{prop:fixed-locus-symplectic-variety}
Let $\iota$ be a 
symplectic automorphism of finite order acting on a symplectic variety $M$ and let $Z\subset\fix(\iota)$ be an irreducible component of its fixed locus.
Suppose that $Z$ is not contained in the singular locus $\Sigma$ of $M$. Then the normalization $\mathcal{Z}$ of  $Z$ is a symplectic variety.
\end{proposition}
\begin{proof} 
Let $\nu\colon \mathcal{Z}\rightarrow Z$ be the normalization morphism. 
By \cite[Proposition 2.6]{Fujiki}, there is a symplectic form \(\sigma_{W}\) on $W:=Z \cap (M\setminus\Sigma)$, obtained as the restriction of the symplectic form $\sigma_M$ on $M$. Hence there is a symplectic form $\sigma'$ on $\nu^{-1}(W)\subset \mathcal{Z}_{\reg}$. We claim that the codimension of the complement of $W$ in $Z$ (and hence also in $\mathcal{Z}$) is greater than or equal to $2$. 
Assuming the claim, it follows by  codimension reasons and normality that $\sigma'$ extends to a holomorphic form $\sigma\in H^0(\mathcal{Z}_{\reg},\Omega^{2}_{\mathcal{Z}_{\reg}} )= H^0(\mathcal{Z},\Omega^{[2]}_{\mathcal{Z}})$. Again by codimension reasons, $\sigma$ has to be non-degenerate on $\mathcal{Z}_{\reg}$. 

To prove the claim, we show that $Z \cap \Sigma_{\reg}$ has a holomorphic 2-form which is generically non-degenerate; this implies that $Z \cap \Sigma_{\reg}$ is even dimensional and hence that $Z \cap \Sigma_{\reg}$ has codimension greater than or equal to $2$ in $Z$.
By Corollary \ref{cor_kaledin}, $\iota$ induces an automorphism $\iota_\Sigma$ on $\Sigma$; moreover, $Z \cap \Sigma$ is a union of components of $\fix(\iota_\Sigma)$.
By Proposition \ref{prop_kaledin} the smooth locus of $\Sigma$ has a holomorphic $2$-form which is generically non-degenerate, and by Corollary \ref{cor_kaledin}, the automorphism $\iota_\Sigma$ preserves this form. As a consequence, every component of $\Fix(\iota_\Sigma) \cap \Sigma_{\reg}$ has a generically non-degenerate holomorphic $2$-form. 
By Remark \ref{rmk_kaledin}, the same argument works for the intersection of $Z$ with any other stratum of the singular locus, so the claim is proved.

We now show that the symplectic form $\sigma\in H^0(\mathcal{Z},\Omega^{[2]}_{\mathcal{Z}})$ extends to a holomorphic form on a resolution of the singularities of $Z$ (and hence of $\mathcal{Z}$).
Consider an embedded resolution of singularities of $Z$ in $M$, i.e., consider
 resolutions \(\widetilde{Z} \to Z\) and \(\pi\colon \widetilde{M}\to M\) of $Z$ and $M$, respectively, with a compatible embedding \(\widetilde{Z} \hookrightarrow \widetilde{M}\).
Since $M$ is a symplectic variety,  the symplectic form on its smooth locus  extends to a holomorphic 2-form on $\widetilde{M}$, which restricts to a holomorphic 2-form on $\widetilde Z$ extending $\sigma_{\reg}$. Since the resolution $\widetilde Z \to Z$ factors via the normalization morphism $\mathcal{Z} \to Z $, the symplectic form $\sigma'$ also extends.
\end{proof}

As a corollary we obtain that when $H=D$, the relative Prym variety $\mathcal{P}_D$ introduced in Definition~\ref{def_Prym} is a symplectic variety in the sense of Definition~\ref{def_symplvariety}:

\begin{proof}[Proof of Proposition~\ref{prop_simplvarBeauville}]
As explained after Definition~\ref{def_Prym}, when $H=D$, the locus $\fix^0(\tau)$ is closed.
Hence, the statement is an immediate consequence of Propositions~\ref{prop-M-symplectic-variety} and \ref{prop:fixed-locus-symplectic-variety}.
\end{proof}

In the case where $H\neq D$, restricting to the biggest locus of $M_H$ where $\tau$ is defined, we still find that $\fix^0(\tau)$ has a holomorphic 2-form which is non-degenerate on $\fix^0(\tau)_{\reg}$.
However, the induced 2-form $\sigma$ on $\cP_H$ can degenerate along divisors in $(\mathcal{P}_H)_{\reg}$ \cite[Example 3.3.7]{Sacca}.

It follows from Definition \ref{def_lagfib} and Proposition \ref{prop_simplvarBeauville} that the map $\eta\colon \cP_D\rightarrow |C|$ is a Lagrangian fibration. A very small modification of a celebrated result of Matsushita implies that, even in the case of symplectic varieties, Lagrangian fibrations are flat:

\begin{proposition}\label{prop_flatness}
    The morphism $\eta\colon\mathcal{P}_D\to |C|$ is equidimensional, and thus flat.
\end{proposition}
\begin{proof}  Let $\rho\colon\widetilde{\mathcal{P}}_D \to \mathcal{P}_D$ be a resolution of singularities.
    Since $\mathcal{P}_D$ is symplectic and hence has trivial canonical bundle and rational singularities, $R\rho_* \omega_{\widetilde{\mathcal{P}}_D }=\omega_{{\mathcal{P}}_D }=\mathcal{O}_{\mathcal{P}_D}$. Hence, we can apply the same argument as in \cite{Matsushita-equidimensionality} to prove equidimensionality. Flatness follows from miracle flatness.
\end{proof}

\begin{remark} \label{rmk_indpol} Recall that the open subset inside the Beauville--Mukai system $M_H$ parametrizing sheaves whose support is integral is independent of the polarization. Moreover, the involution $\tau$ is regular on this locus. 
    Thus, letting $U' \subset |C|$ be the locus parametrizing curves whose preimage is integral, the open subset $\eta^{-1}(U') \subset \cP_H$ is independent of $H$.
\end{remark}

\subsection{Proof of Theorem \ref{thm_psv} and a criterion for being an irreducible symplectic variety}\label{sec:psv}

Our next goal is to understand when $\mathcal{P}_D$ is a primitive or irreducible symplectic variety in the sense of Definitions \ref{def_psv} and \ref{def_irreduciblesymplectic}. The key ingredient for the next two statements is the following result:

\begin{proposition}[{\cite[Proposition 5.8]{Kebekus}}]
\label{prop_kebekus}
Let $g\colon Z\to X$ be a dominant morphism between klt normal irreducible varieties. Then the reflexive pullback
$g^{[*]}\colon H^0(X,\Omega_X^{[p]})\to H^0(Z,\Omega_Z^{[p]})$
is injective.
\end{proposition}

\begin{proposition}\label{prop_condition_for_psv}
Let $P$ be a projective symplectic variety. Suppose that there exists a dominant rational map $\varphi \colon M\dashrightarrow P$, where $M$ is an irreducible symplectic variety. Then $P$ is a primitive symplectic variety.
\end{proposition}
\begin{proof}
Let $\rho \colon \widetilde{P}\rightarrow P$ be a resolution of singularities, with $\widetilde{P}$  projective.
The map $M \dashrightarrow P$ induces a dominant rational map ${M} \dashrightarrow \widetilde{P}$. Let $\widetilde{M}$ be a resolution of the indeterminacies of this map, with $\widetilde{M}$ smooth. We have the commutative diagram
\begin{equation} \label{diagramres}
\xymatrix{\widetilde{M}\ar[d] \ar[r] & \widetilde{P}\ar[d]^{\rho} &\hspace{-21pt} \\
M\ar@{-->}[r] & P &
}    
\end{equation}
Note that $\widetilde{M}$ is birational to the irreducible symplectic variety $M$, and therefore $h^1(\widetilde M, \mathcal{O}_{\widetilde M})=0$. Since $\widetilde{M} \to \widetilde{P}$ is dominant, we have $h^1(\widetilde{P},\cO_{\widetilde{P}})=0$, and hence $h^1({P},\cO_{{P}})=0$, because $P$ has rational singularities.
The proposition follows from applying Proposition \ref{prop_kebekus} to the composition $\widetilde{M} \to P$. Indeed, since $P$ has canonical singularities and this composition is dominant, the pullback of differential forms is injective, so 
$h^0(P,\Omega^{[2]}_P)\leq  h^0(M,\Omega^{[2]}_M)=1$. 
\end{proof}

In order to show that $\mathcal{P}_D$ is an irreducible symplectic variety, we need the following general result, which was used in \cite{PR} to show that singular moduli spaces of semistable sheaves on K3 surfaces are irreducible symplectic varieties.

\begin{proposition}
\label{prop_keyprop}
Let $P$ be a projective symplectic variety with symplectic form $\sigma$.
Suppose that: 
\begin{enumerate}[1),itemsep=0pt]
\item There is a dominant rational map $M\dashrightarrow P$, where $M$ is an irreducible symplectic variety; 
\item $\pi_1(P_{\reg})=1$.
\end{enumerate}
Then $P$ is an irreducible symplectic variety.
\end{proposition}
\begin{proof} Let $\psi\colon X\to P $ be a finite quasi-\'etale morphism. We need to show that the exterior algebra of reflexive holomorphic forms on $X$ is spanned by $\psi^{[*]}\sigma$.
By purity of the branch locus, the restriction $\psi^{-1}(P_{\reg})\to P_{\reg}$ is \'etale. 
Since $\pi_1(P_{\reg})=1$, it follows that $\psi^{-1}(P_{\reg})\to P_{\reg}$ is an isomorphism. 

It follows that $\psi$ is birational and, since it is finite and $P$ is normal, that it is in fact an isomorphism.
As a consequence, we are reduced to checking that the algebra of reflexive forms on $P$ is generated by the symplectic form $\sigma$. 
But this follows as in the proof of the Proposition \ref{prop_condition_for_psv}, using the composition $\widetilde M \to P$, the fact that $\widetilde{M}$ is birational to the irreducible symplectic variety $M$, and  Proposition \ref{prop_kebekus}. 
\end{proof}

Let us come back to the relative Prym variety $\mathcal{P}_D$ defined as in Definition~\ref{def_Prym}. Our first application of the criterion establishes in Proposition~\ref{prop_keyprop} above is to show that the (normalization of the) relative Prym varieties of \cite[Theorem 1.1]{ASF} are irreducible symplectic varieties:
\begin{corollary}\label{cor_ASFisISV}
Let $T$ be a general Enriques surface with a linear system $|C|$, let $f\colon S \to T$ be its K3 double cover, and let $|D|$ the pullback of $|C|$ to $S$. The associated relative Prym variety $\mathcal{P}_D$ is an irreducible symplectic variety.
\end{corollary}
\begin{proof}
The two requirements of Proposition~\ref{prop_keyprop} are satisfied by \cite[Theorem 7.1 and proof of Theorem 8.1]{ASF}
\end{proof}

We now show that if $C$ and $D$ are very ample, then condition (1) in Proposition~\ref{prop_keyprop} holds. 
We will use this to prove Theorem~\ref{thm_psv}.

\begin{proposition}\label{prop_dominantMap}
Let $S$ be a  smooth K3 surface with an anti-symplectic involution $i$ and let $f\colon S\rightarrow T=S/i$ be the quotient map. Assume that $(S,i)$ is very general in the sense of Definition \ref{def_generalK3withsympinv}. Let $D$ be a smooth curve of genus $h$ on $S$ and $C$ a smooth curve of genus $g$ on $T$ such that $f^{-1}C=D$. If the linear systems $|C|$ and $|D|$ are very ample, then for any polarization $H$ on $S$ there is a dominant rational map $$\varphi\colon S^{[h-g]} \dashrightarrow \mathcal{P}_H.$$
\end{proposition}
\begin{proof}
We prove the proposition for $H=D$.
Since for any choice of a polarization $H$ on $S$, $\cP_H$ and $\cP_D$ are birational, the statement holds for any $\cP_H$. 

Set $V:=H^0(S, D)^\vee$.
Since $D$ is very ample, we have an embedding 
$$S \hookrightarrow \mathbb{P}(V) \cong \mathbb{P}^{h}.$$
Recall that the branch divisor $B$ satisfies $B=-2K_T$.
In particular, $f$ is branched along a section of $\omega_T^{\vee \otimes 2}$, so $f_*\cO_S= \cO_T \oplus \omega_T$. Then by the projection formula, it follows that  
\[f_*\cO_S(D)= f_*(f^*\cO_T(C) \otimes \cO_S)= \cO_T(C) \oplus \cO_T(C + K_T).\] 
Since $f$ is finite, we have
\[H^0(S, D) = H^0(T, f_*D)= H^0(T, C) \oplus H^0(T, C+K_T).\]
As $C$ is very ample, we have an embedding 
\[T \hookrightarrow \p(H^0(T, C)^\vee) \cong \mathbb{P}^{h-g}.\]
Then the involution $i$ extends to an involution of $\p(V)$ and $f$ is the restriction of the projection from $\p(H^0(T, C+K_T)^\vee)$. 
Denote by $V_+$ and $V_-$ the eigenspaces of the involution on $\mathbb{P}(V)$ corresponding to $1$ and $-1$.
Up to changing the linearization, we can assume $V_-=H^0(T, C)^\vee$, $V_+=H^0(T, C+ K_T)^\vee$.
We will construct a dominant rational map $\varphi\colon S^{[h-g]}\dashrightarrow \mathcal{P}_D$ fitting in the diagram
\begin{equation}
\label{cd_defofvarphi}    
\xymatrix@R=1mm{S^{[h-g]}\ar@{-->}[rr]^{\varphi}\ar@{-->}[ddr] & & \mathcal{P}_D\ar[ddl]^\eta\\
\!\!\!Z\ar@{|->}[ddr] & & \\
 & |C| \cong \p(V_-) & \\
 & \!C_Z &
}.
\end{equation}
Here $C_Z$ is defined as follows.
Consider the subset $U \subset S^{[h-g]}$ consisting of elements $Z=\{z_1, \ldots, z_{h-g}\} \in S^{[h-g]}$, where the $z_i$'s are distinct points on $S$ and such that:

\begin{enumerate}[1.]
    \item \(H_{f(Z)}:=\langle f(z_1), \ldots, f(z_{h-g})\rangle\) is a hyperplane in \(\mathbb{P}^{h-g}\);
    \item $Z \cap i(Z)=\emptyset.$ More precisely $i(z_k) \neq z_l$ for all $k, l\in \{1, \ldots, h-g\}$.
\end{enumerate}
We claim that $U$ is an open subset of $S^{[h-g]}$.
Indeed, let $W$ be the image of $U$ under the Hilbert--Chow morphism $HC\colon S^{[h-g]}\to \Sym^{h-g}(S)$, in particular, $HC^{-1}(W)=U$.
Set $W':=q^{-1}(W)$, where $q \colon S^{ h-g} \to \Sym^{h-g}(S)$ is the quotient map. 
Note that the complement of $W'$ in $S^{ h-g}$ is  $Y_1\cup Y_2$, where:
\begin{enumerate}[a.]
    \item $Y_1$ is the set consisting of $(h-g)$-tuples of points $z_1, \dots, z_{h-g}$ in $S$ such that $H_{f(Z)}$ is not a hyperplane in $\mathbb{P}^{h-g}$.
    \item $Y_2$ is the set consisting of $(h-g)$-tuples of points $z_1, \dots, z_{h-g}$ in $S$ 
    such that for some $k,l$ the pair $(z_k, z_l)$ belongs to the graph of the involution $i$.
\end{enumerate}
 Note that $Y_1$ and $Y_2$ are closed in $S^{ h-g}$, so $W'$ is open. We conclude that $W$ is open in the quotient, hence $U$ is open in $S^{[h-g]}$.

As a consequence, since $U$ dominates $|C|$, by Bertini's Theorem for general $Z \in U$ the curve
\[C_{Z}:=H_{f(Z)} \cap T\in|C|\] 
is smooth.
Define $D_{Z}:=f^{-1}(C_Z)=\langle \mathbb{P}(V_{+}),Z \rangle \cap S$.
Note that the locus of curves in $|C|$ intersecting $B$ non transversely is parametrized by the dual varieties of the irreducible components of $B$, so it has codimension $\geq 1$. Hence a general smooth curve $C_Z \in |C|$ intersects $B$ transversely, so $D_Z$ is smooth.

Then $\varphi\colon S^{[h-g]}\dashrightarrow \mathcal{P}_D$ is defined by 
$$U\ni Z \mapsto \mathcal{O}_{D_Z}(Z-i^*Z),$$ 
where $Z$ is considered as a divisor on $D_Z$. The sheaf $\mathcal{O}_{D_Z}(Z-i^*Z)$ is $\tau$-invariant by definition and stable by \cite[Lemma~1.1.13]{Sacca}, so it is an element in $\fix(\tau)$. The curve $C_Z$ is general, hence the irreducible component $K$ of $\fix(\tau)$ that contains $\prym(D_Z/C_Z)$ dominates the linear system $|C|$. 
The dimension of $K$ equals the dimension of $\fix(\tau)$, hence by Corollary~\ref{cor_Fix=connected} it follows that $K$ coincides with $\fix^0(\tau)$. Moreover, since $\mathcal{O}_{D_Z}(Z-i^*Z)$ is stable, it belongs to the regular part of $\fix^0(\tau)$, so it defines a point in the normalization $\mathcal{P}_D$.

We claim that $\varphi$ is dominant. 
Indeed, first consider the map $S^{[h-g]} \dashrightarrow |C|$ defined by $Z \mapsto C_Z$.
For general smooth curves $C_0 \in |C|$ and $D_0=f^{-1}C_0$, its fiber over $C_0$ is an open subset of $\Sym^{h-g}D_0$ (note that the map factors over $f^{[h-g]} \colon S^{[h-g]} \dashrightarrow T^{[h-g]}$).
Since $\varphi$ commutes with the fibrations on $\p(V_-)$ in \eqref{cd_defofvarphi}
and $\dim \mathcal{P}_D=2h-2g=\dim S^{[h-g]}$, in order to show that $\varphi$ is dominant it is enough to show that the rational map
\[\psi \colon \Sym^{h-g}D_0 \dashrightarrow \mathcal{P}_D; \quad Z_0:=x_1+\dots +x_{h-g}\mapsto \mathcal{O}_{D_0}(Z_0- i^*Z_0)\]
has image of dimension $h-g$. 
Indeed, this would imply that the image of $\varphi$ has dimension
$$\dim(\text{Im}(\psi))+ \dim(|C|)=2h-2g.$$
Note that, by definition, $\psi$ is the composition of the Abel--Jacobi map $u \colon \Sym^{h-g}D_0 \to J(D_0)$ and the projection $1-i^* \colon J(D_0) \to \prym(D_0/C_0)$. Consider distinct points $x_1, \dots, x_{h-g}$ in $D_0$ defining an element in $U$.
The rank of the differential $u_*$ of $u$ at the point $Z_0=x_1+ \dots+ x_{h-g}$ is equal to the rank of the matrix $(\omega_k(x_l))$, $k=1, \dots, h$, $l= 1, \dots, h-g$, where $\omega_1, \dots, \omega_h$ are a basis of $H^0(D_0, \omega_{D_0})$. 
Moreover, we can choose a basis such that $\omega_1, \dots, \omega_{g}$ are $i$-invariant, while $\omega_{g+1}, \dots, \omega_{h}$ are $i$-anti-invariant. 
Then the rank of $\psi_*$ is equal to the rank of the $(h-g) \times (h-g)$-matrix $(\omega_k(x_l))$, $k=g+1, \dots, h$, $l= 1, \dots, h-g$. This matrix has maximal rank for general $Z_0$, as otherwise the linear span of $x_1, \dots, x_{h-g}$ would intersect $\p(V_+)$, in contradiction with the generality assumption. 
\end{proof}

\begin{remark}\label{rmk_rationalmap_Cample} In the proof of Proposition \ref{prop_dominantMap} above, very ampleness of $|C|$ is needed in order to ensure that, for general $Z$, the curves $C_Z$ and $D_Z$ are irreducible; as a byproduct this implies that all the irreducible components of the branch divisor $B$ are contained in $\mathbb{P}(V_-)$. 
A similar proof, similar to the ones of \cite[Lemma 5.2]{MT} and \cite[Section 4.4.2]{Matteini}, gives the statement of Proposition~\ref{prop_dominantMap} also in the case when $C$ is ample and $B$ is irreducible, i.e.\ when the main invariant of $i$ satisfies $r=a$. 
\end{remark}

\begin{remark}
The degree of the map in Proposition~\ref{prop_dominantMap} turns out to be two in the two known cases in the literature (see the aforementioned \cite[Lemma 5.2]{MT} and \cite[Section 4.4.2]{Matteini}). Our expectation is that the degree of $\varphi$ is 
\[1+\left(\begin{array}{c}
  (2h-2)-(h-g)  \\
     h-g
\end{array}\right).\]
We also remark that the case $g=0$ is excluded by the assumptions of Proposition \ref{prop_dominantMap}, since in that case $D$ would be hyperelliptic, hence not very ample.
\end{remark}

We finish the section by using Proposition~\ref{prop_dominantMap} to prove Theorem~\ref{thm_psv}.

\begin{proof}[Proof of Theorem \ref{thm_psv}]
This is a consequence of Propositions~\ref{prop_simplvarBeauville},  \ref{prop_condition_for_psv} and \ref{prop_dominantMap}. 
\end{proof}

\section{Simple connectedness}\label{sec:pi1}

In the section we focus on the second condition of Proposition~\ref{prop_keyprop}, the criterion that will allow us to prove Theorem \ref{thm_isv_exactHp}.
The main result of this section is Theorem~\ref{thm_pi1}, 
which proves simple connectedness of the regular locus of the relative Prym varieties $\mathcal{P}_H$. In Sections \ref{section_curvesreduciblepreim} and \ref{section_curvessingularpreim}, we will express the technical conditions (i)--(iii) below in terms of positivity properties of the linear system $|C|$. Finally, in the last Section \ref{sec:examples}, we will give examples of linear systems which satisfy these conditions.

\begin{theorem}\label{thm_pi1}
Let $S$ be a very general K3 surface with an anti-symplectic involution $i$; let $f\colon S\to T:=S/i$ be the quotient map. Let $C$ be a curve on $T$ and assume that
\begin{enumerate}[(i)]
 \item The branch locus $B$ satisfies $B.C>2$;
 \item The locus of curves in $|C|$ whose preimage under $f$ is not integral has codimension at least 2 in $|C|$;
 \item 
 Let $Z\subset|C|$ be an irreducible component of the locus of curves with singular preimage for which $\codim_{|C|}Z=1$.
Then the general element of $Z$ is one of the following:
 \begin{enumerate}[(a)]
    \item[(1)] A smooth integral curve intersecting $B$ transversely except in one point, where the multiplicity is 2;
    \item[(2)] An integral curve intersecting $B$ transversely, with one node outside $B$ and no other singularities.
 \end{enumerate}
\end{enumerate}
Then for any polarization $H$ on $S$, the smooth locus ${(\cP_H)}_{\reg}$ of the associated relative Prym variety $\cP=\cP_H$ is simply connected. 
\end{theorem}

We follow the proof of Theorem~7.1 in \cite{ASF}.
We will suppress the base point from the notation of fundamental groups.

\medskip
Throughout this section, we denote by $U'\subset|C|$ the locus of curves whose preimage under $f$ is integral. 
So by assumption, the complement $|C|\setminus U'$ has codimension at least 2 in $|C|$.
By $U\subset U'$ we denote the locus of curves in $|C|$ whose preimage is integral and smooth.
The complement $W:=U'\setminus U$ consists of curves in $|C|$ whose pullback under $f$ is integral and singular. 
By assumption, the irreducible components $W_i$ of $W$ that have codimension 1 in $U'$ come in two types, depending on what their general element looks like:

\begin{enumerate}
 \item We say $W_i$ is of \emph{type 1} if its general point corresponds to a smooth curve $C$ in $T$ which intersects $B$ transversely except in one point, where it has multiplicity 2.
 \item We say $W_i$ is of \emph{type 2} if its general point corresponds to a curve $C$ with one node away from the branch locus $B$ and no other singularities, and which intersects $B$ transversely. 
\end{enumerate}

One can determine the type of singularities of the curve $f^{-1}C$ using the following lemma:

\begin{lemma}\label{lemma_milnornumber}
Let $p\in C$ be a smooth point such that $C$ intersects $B$ at $p$ with multiplicity $n$. Then the point $q$ of $D$ above $p$ has Milnor number $n-1$.
\end{lemma}
\begin{proof}
 Locally, we can assume the degree 2 cover is given by
 \begin{align*}
   \spec\mathbb{C}[x,y,z]\supset Z(z^2-y)&\to \spec\mathbb{C}[x,y]  \\
  (x,y,z)&\mapsto (x,y),
\end{align*}
with branch curve the $x$-axis. We may assume $C$ is the curve $Z(y-x^n)$ and $p=(0,0)$. 
 The inverse image of $Z(y-x^n)$ is $Z(z^2-y,y-x^n)$, 
 which is isomorphic to $Z(z^2-x^n)\subset\spec\mathbb{C}[x,z]$.
 Its singular point $q$ has Milnor number
 \[\mu(q)=\dim_{\mathbb{C}}\frac{\mathbb{C}[[x,z]]}{(2z,nx^{n-1})}=n-1.\qedhere\]
\end{proof}

It follows that for a general curve $C$ in a branch $W_i$ of type 1, 
the curve $f^{-1}C\subset S$ has one node and no other singularities.
If $C$ is a general curve in a branch $W_i$ of type 2, then $f^{-1}C\subset S$ has two nodes and no other singularities.

\subsection{First reduction step}

Denote by $\eta\colon \cP\to |C|$ the map induced by the support morphism.
The inverse image $\cP':=\eta^{-1}(U')$ is contained in $\cP_{\reg}$ (see Remark~\ref{remark_sectionsupportP}). 
The pushforward map $\pi_1(\cP')\to\pi_1(\cP_{\reg})$ is surjective by \cite[Proposition~2.10]{Kollar}; hence, it suffices to prove that $\pi_1(\cP')=1$. To do this, we will use Leibman's theorem \cite[Section~0.3]{Leibman}:
\begin{theorem}[Leibman]
    Let $p\colon E\to B$ be a surjective morphism of connected smooth manifolds.
    Suppose there exists an open subset $V\subset B$ such that the restriction
    $E_V\to V$ is a locally trivial fibre bundle with fibre $F$, and such that $B\setminus V$  has real codimension at least 2 in $B$.
    In addition, suppose that $p$ has a global section $s\colon B\to E$ whose image intersects every irreducible component of $E\setminus p^{-1}(V)$ of real codimension 2 in $E$. 
    Consider the short exact sequence
    \[1\to \pi_1(F)\to \pi_1(E_V)\overset{s_*}{\leftrightarrows} \pi_1(V)\to 1.\]
    Let $K:=\ker(\pi_1(V)\to \pi_1(B))$, considered as a subgroup of $\pi_1(E_V)$ via $s_*$. Let $R=[\pi_1(F),K]$ be the commutator subgroup of $\pi_1(F)$ and $K$ in $\pi_1(E_V)$.
    Then there exists an exact sequence
       \[ 1\to R\to \pi_1(F)\to \pi_1(E)\overset{s_*}{\leftrightarrows} \pi_1(B)\to 1 \]
\end{theorem}

We apply Leibman's theorem to the family $\cP'\to U'$. Indeed, this has a section $s\colon U'\to \cP'$, given by sending $C\in U'$ to the pushforward of the structure sheaf $\mathcal{O}_{f^{-1}C}$ to $S$ (see also \eqref{zerosection_supportmap}). Moreover, we will now show that the section $s(U')$ intersects every irreducible component of $\cP'\setminus \eta^{-1}(U)$ of real codimension $2$ in $\cP'$.

\begin{lemma}
\label{lemma_generic-irreducibility_fibers_codim1}
    Let $C_0$ be a general point in an irreducible component of $W=U'\setminus U$ of codimension 1 in $U'$.
    Then the fibre $\cP'_{C_0}$ is irreducible.
\end{lemma}
\begin{proof}
We use the notation of Proposition-Definition~\ref{involution_integral_case}. 
The fibre $M_H(v)_{C_0}$ is the compactified Jacobian $\overline{J(D_0)}$ of $D_0=f^*C_0$,
which has induced involution $\tau_0=\tau|_{\overline{J(D_0)}}$.
The fibre $\cP'_{C_0}=\fix^0(\tau)_{C_0}$ consists of all irreducible components of $\fix(\tau_0)$ contained in $\fix^0(\tau)$.
By uppersemicontinuity of fiber dimension, all components of $\cP'_{C_0}$ have dimension at least $g_a(D_0)-g_a(C_0)$, the dimension of the fibre over a smooth curve $C\in U$.
By Lemmas~\ref{integral Prym}, \ref{irr_prym_one_node} and \ref{irr_prym_transv_B}, $\fix(\tau_0)$ has only one irreducible component of maximal dimension $g_a(D_0)-g_a(C_0)$, which implies the statement.
\end{proof}

\begin{proposition}
    Under the assumptions of Theorem \ref{thm_pi1}, the section $s(U')$ intersects every irreducible component of $\cP'\setminus \eta^{-1}(U)$ of codimension $1$ in $\cP'$.
\end{proposition}
\begin{proof}
    By Proposition \ref{prop_flatness} and Remark \ref{rmk_indpol}, the restriction $\cP' \to U'$ of $\eta$ to $U'$ is flat. Moreover, Lemma \ref{lemma_generic-irreducibility_fibers_codim1} shows that over the general point $C_0$ in any codimension-1 component of $U'\setminus U$, 
    the fibre $\cP'_{C_0}$ is irreducible.
    As a consequence, denoting by 
    $\cP'_{\mathrm{irr}}\subset\cP'$ the open subset of irreducible fibers, the complement $\cP'\setminus \cP'_{\mathrm{irr}}$ has codimension $\geq 2$, and this yields the statement. 
\end{proof}

Over the locus $U\subset U'$, the family $\cP_U:=\eta^{-1}(U)\to U$ is a locally trivial fibration of real manifolds,
with fibre $P_0:=\prym(D/C)$ for any curve $C$ in $U$ with preimage $D$. 
Hence we have exact sequences
\[1\to \pi_1(P_0)\to\pi_1(\cP_U)\overset{s_*}{\leftrightarrows} \pi_1(U)\to 1\]
and, by Leibman's theorem,
\begin{equation}\label{sequence_Leibman}
1\to R\to \pi_1(P_0)\to\pi_1(\cP')\overset{s_*}{\leftrightarrows} \pi_1(U')\to 0.
\end{equation}
Here $R$ is the commutator subgroup
\[[\pi_1(P_0),K]\subset\pi_1(\cP_U)\]
where $K=\ker(\pi_1(U)\to \pi_1(U'))$, considered a subgroup of $\pi_1(\cP_U)$ via $s_*$.
By assumption, the complement of $U'$ in $|C|$ has complex codimension at least 2. 
Hence, $U'$ is simply connected and $K=\pi_1(U)$, so the exact sequence \eqref{sequence_Leibman} reduces to
\[1\to [\pi_1(P_0),\pi_1(U)]\to \pi_1(P_0)\to \pi_1(\cP')\to 1.\]
In order to show $\pi_1(\cP')=1$, it thus suffices to show that 
$[\pi_1(P_0),\pi_1(U)]\to \pi_1(P_0)$ is surjective.
Note that we can identify $\pi_1(P_0)$ with $H_1(D,\mathbb{Z})_-\subset H_1(D,\mathbb{Z})=\pi_1(J(D))$.
Indeed, by \cite[Theorem 3.1.2]{Matteini}, \cite[(3.5)]{LO} we have
$P_{0}={H^{0}(\Omega_D)}^*_{-}/H_{1}(D,\mathbb{Z})_{-}$. 
Below, we will give generators for both $\pi_1(U)$ and $H_1(D,\mathbb{Z})_-$.

\subsection{Generators of the anti-invariant homology}

Since $U'$ is simply connected, the group $\pi_1(U)$ is generated by simple loops around the irreducible branches of $W=U'\setminus U$ of codimension 1. Formally, one may find such loops as follows. We fix a generic line $\ell\subset|C|$ 
contained in $U'$, such that $\ell$ intersects $W$ transversely. 
Then the pushforward map
\[\pi_1(\ell\setminus (\ell\cap W))=\pi_1(\ell\cap U)\to \pi_1(U)\]
is surjective by \cite[Theorem~1.1.B]{FuLaz}. 
The group $\pi_1(\ell\setminus (\ell\cap W))$ is generated by classes of simple loops around the points in $\ell\cap W$; the classes of the pushforwards of these loops in $U$ generate $\pi_1(U)$.

Denote by $V'\subset |D|$ the locus of integral curves in $|D|$, and by $V\subset V'$ the locus of smooth integral curves.
So the pullback $g:=f^*\colon |C|\hookrightarrow |D|$ sends $U'$ into $V'$ and $U$ into $V$. We will find the image of the homomorphism $g_*\colon \pi_1(U)\to \pi_1(V)$ by describing the image of a set of generators of $\pi_1(U)$, and eventually use this to understand $\pi_1(V)$ and $H_1(D,\mathbb{Z})_-$. We will need to distinguish the generators of $\pi_1(U)$ coming from branches of type 1 and 2. 
Let us denote by $\{x_r\}_{r\in\mathfrak{r}}$ the set of points in $\ell\cap W$ which lie on a branch of $W$ of type 1; the corresponding loops in $U$ around these branches we denote by $\{\gamma_r\}_{r\in\mathfrak{r}}$. 
Similarly, denote by $\{x_s\}_{s\in\mathfrak{s}}$ and $\{\gamma_s\}_{s\in\mathfrak{s}}$ the points and loops corresponding to branches of type 2. 
So the system of generators of $\pi_1(U)$ that we will work with is
\[\mathcal{G}:=\{[\gamma_r],[\gamma_s]\mid r\in\mathfrak{r},s\in\mathfrak{s}\}.\]

Note that $(|D|\setminus V')\cap g(|C|)$, is identified via $g$ with the locus of curves in $|C|$ with non-integral preimage, whose codimension in $|C|$ is at least 2 by assumption.
It follows that $|D|\setminus V'$ has codimension at least 2 in $|D|$. 
Hence $V'$ is simply connected and $\pi_1(V)$ is generated by simple loops around the codimension 1 irreducible branches of $W^D:=V'\setminus V$.
This set consists of integral curves in $|D|$ which are singular.

\begin{lemma}
The general element of $W^D$ is a curve with one node and no other singularities.
\end{lemma}
\begin{proof}
If $Z\subset W^D$ is an irreducible component of codimension 1 in $|D|$ and $g'$ is the geometric genus of a general curve $D'$ in $Z$, then $Z$ is contained in the Zariski closure of the locus of irreducible curves in $|D|$ of geometric genus $g'$.
Then by \cite[Proposition~4.5]{DedieuSernesi}, we have $g'=g(D)-1$. 
This means that $D'$ has one singularity with $\delta$-invariant $\delta=1$, that is, a node or a cusp.
By \cite[Proposition~4.6]{DedieuSernesi} $D'$ is immersed, meaning that the differential of its normalization morphism is everywhere injective.
Therefore $D'$ cannot have a cusp.
\end{proof}

For an element $[\gamma]\in \mathcal{G}$, the  image $g_*[\gamma]$ under $g_*\colon \pi_1(U)\to \pi_1(V)$ depends on the type of branch that $\gamma$ is associated with. 
We first consider a simple loop $\gamma_r$ around a branch $W_{\gamma_r}$ of the first type. So the general element of $W_{\gamma_r}$ corresponds to a curve $C_0\in|C|$ whose preimage $D_0=f^{-1}C_0\in |D|$ has one node and no other singularities. 
By \cite[Example~1.3]{CS} the point $[D_0]\in |D|$ is a smooth point of $W^D$.
Let $W^D_{\gamma_r}$ be the irreducible branch of $W^D$ that $[D_0]$ lies in.
Then $g_*[\gamma_r]=[\lambda_r]$ for some simple loop $\lambda_r$ around $W^D_{\gamma_r}$.

Next, consider a loop $\gamma_s$ around a branch $W_{\gamma_s}$ of the second type. So the general element of $W_{\gamma_s}$ corresponds to a curve $C_0$ whose preimage
$D_0=f^{-1}C_0$ has two nodes and no other singularities.
The point $[D_0]$ lies in the transverse intersection of two irreducible branches $W^D_{\gamma_s,1}$ and $W^D_{\gamma_s,2}$ of $W^D$ (see \cite[Example~1.3]{CS}).
As in \cite[proof of Theorem~7.1]{ASF}, these branches are interchanged by the involution $i$, and $g_*[\gamma_s]=[\lambda_s][i\lambda_s]$ for some loop $\lambda_s$ around, say, $W^D_{\gamma_s,1}$.

\medskip
Summarizing the above, we have found
\begin{align*}
    g_*[\gamma_r]&=[\lambda_r]\\
    g_*[\gamma_s]&=[\lambda_s][i\lambda_s]
\end{align*}
where $\lambda_r$ and $\lambda_s$ are 
simple loops around irreducible branches of $W^D$.
Similarly as in \cite{ASF}, we can show:
\begin{proposition}\label{prop_fundgroupV}
 The loops $\{\lambda_r,\lambda_s,i\lambda_s\mid r\in\mathfrak{r},s\in\mathfrak{s}\}$ generate the fundamental group $\pi_1(V)$.
\end{proposition}
\begin{proof}
Recall that in order to define the system of generators $\mathcal{G}$ of $\pi_1(U)$, we fixed a generic line $\ell\subset|C|$ contained in $U'$.
 We move this line $\ell$ slightly in $|D|$ to a line $m$ which intersects $W^D$ transversely. By the observations above, we have 
 \[m\cap W^D=\{y_r,y_s^1,y_s^2\mid r\in\mathfrak{r},s\in\mathfrak{s}\}\]
 where $y_r$ lies in $W^D_{\gamma_r}$ and $y_s^i$ lies in $W^D_{\gamma_s,i}$. 
 Under the map
 \begin{equation}\label{surjection_fundgroep}
  \pi_1(m\setminus \{y_r,y_s^1,y_s^2\})\to \pi_1(V),
 \end{equation}
the class of a simple loop around $y_r$ is mapped to the class of $\lambda_r$, similarly for $\lambda_s$, and the class of a simple loop around $y_s^2$ is mapped to the class of $i\lambda_s$.
Since the map \eqref{surjection_fundgroep} is surjective \cite[Theorem~1.1.B]{FuLaz} and the classes of simple loops around each point of $m\cap W^D$ generate $\pi_1(m\setminus \{y_r,y_s^1,y_s^2\})$, the claim follows. 
\end{proof}

Denote by $\mathcal{D}\to|D|$ the universal family over $|D|$.
As explained in \cite[Section~7]{ASF}, applying Leibman's theorem to the support map $M_H\to |D|$ and using that $M_H$ is simply connected (Proposition~\ref{prop_Msimplyconn}), 
one shows that $\pi_1(J(D))=H_1(D,\mathbb{Z})$ is generated by the vanishing cycles of the family $\mathcal{D}|_V\to V$ corresponding to the loops $\{\lambda_r,\lambda_s,i\lambda_s\}$.
We choose such vanishing cycles and call them
$\{\alpha_r,\alpha_s,i\alpha_s\}$.
Indeed, if $\alpha_s$ is a vanishing cycle for $\lambda_s$, one shows that $i\alpha_s$ is a vanishing cycle for $i\lambda_s$:
if $[C_0]$ is a general point in $W_{\gamma_s}$, take a small disk $\mathbb{D}\subset|C|$ that intersects $W_{\gamma_s}$ transversely at $[C_0]$.
Then $f^{-1}\mathbb{D}$ is a small anti-invariant disk in $|D|$ intersecting $W_{\gamma_s}^D$ transversely at $[D_0]$.
It corresponds to a family of smooth $i$-anti-invariant curves on $S$ degenerating to the curve $D_0$ which has two nodes, interchanged by $i$.
It follows that if $\alpha_s$ is a vanishing cycle for one node, then $i\alpha_s$ is a vanishing cycle for the other node.
Note that the intersection number $[\alpha_s].[i\alpha_s]$ is 0.

\begin{proposition} \label{invariant vanishing cycle}
    The homology classes of vanishing cycles $\{\alpha_r\mid r\in\mathfrak{r}\}$ can be represented by non-separating simple closed curves which are $i$-anti-invariant and smooth.
\end{proposition}

\begin{proof}
    Fix a cycle $\alpha_r$. It corresponds to a loop $\lambda_r$ such that $[\lambda_r]=g_*[\gamma_r]$ where $\gamma_r$ is a loop around a branch $W_{\gamma_r}$ of type 1.
    We can assume that $\gamma_r$ bounds a small disk $\mathbb{D}\subset|C|$ that intersects $W_{\gamma_r}$ transversely at a general point $[C_0]\in W_{\gamma_r}$.
    The family of curves in $|C|$ over $\mathbb{D}$ consists of smooth curves only, whereas the central fiber $D_0$ of the double cover of this family acquires a node $p$. To simplify the notation, we will identify $\mathbb{D}$ with its image $g(\mathbb{D})\subset|D|$.

    Locally near the node $p\in D_0$ and its image $f(p)\in C_0$, we can view the restriction of the universal family $\mathcal{D}\rightarrow |D|$ to $\mathbb{D}$ as $\mathbb{C}^2\rightarrow \mathbb{C},\; (x,y)\mapsto x^2+y^2$, the restriction of the universal family $\mathcal{C}|_\mathbb{D}\rightarrow \mathbb{D}\subset |C|$ as $\mathbb{C}^2\rightarrow \mathbb{C},\; (x,y)\mapsto x^2+y$, 
    and the family of double covers $\mathcal{D}|_{\mathbb{D}}\rightarrow \mathcal{C}|_{\mathbb{D}}$ as $\mathbb{C}^2\rightarrow \mathbb{C}^2,\;(x,y)\mapsto (x,y^2)$,
    with covering involution $(x,y)\mapsto(x,-y)$. 
    The fiber $D_t\subset \mathcal{D}$ over a point $t\in\mathbb{D}$ contains an anti-invariant circle $S(t)=\sqrt{t}\cdot S^1$ where $S^1=\{(x,y)\in \mathbb{R}^2\subset\mathbb{C}^2\mid x^2+y^2=1 \}$. 
    The class of the circle $S(t)\subset D_{t}$ is a vanishing cycle of $p$ by construction. Since $p$ is a node, we have $[S_{t}]=\pm[\alpha_r]$.
    Recall that a simple closed curve on $D_t$ is separating if and only if its class is $0$ in $H_1(D_t,\mathbb{Z})$. Therefore $S(t)$ is an $i$-anti-invariant non-separating simple closed curve as required.
\end{proof}

\begin{figure}
  \centering
  \includegraphics [width=0.9\textwidth]{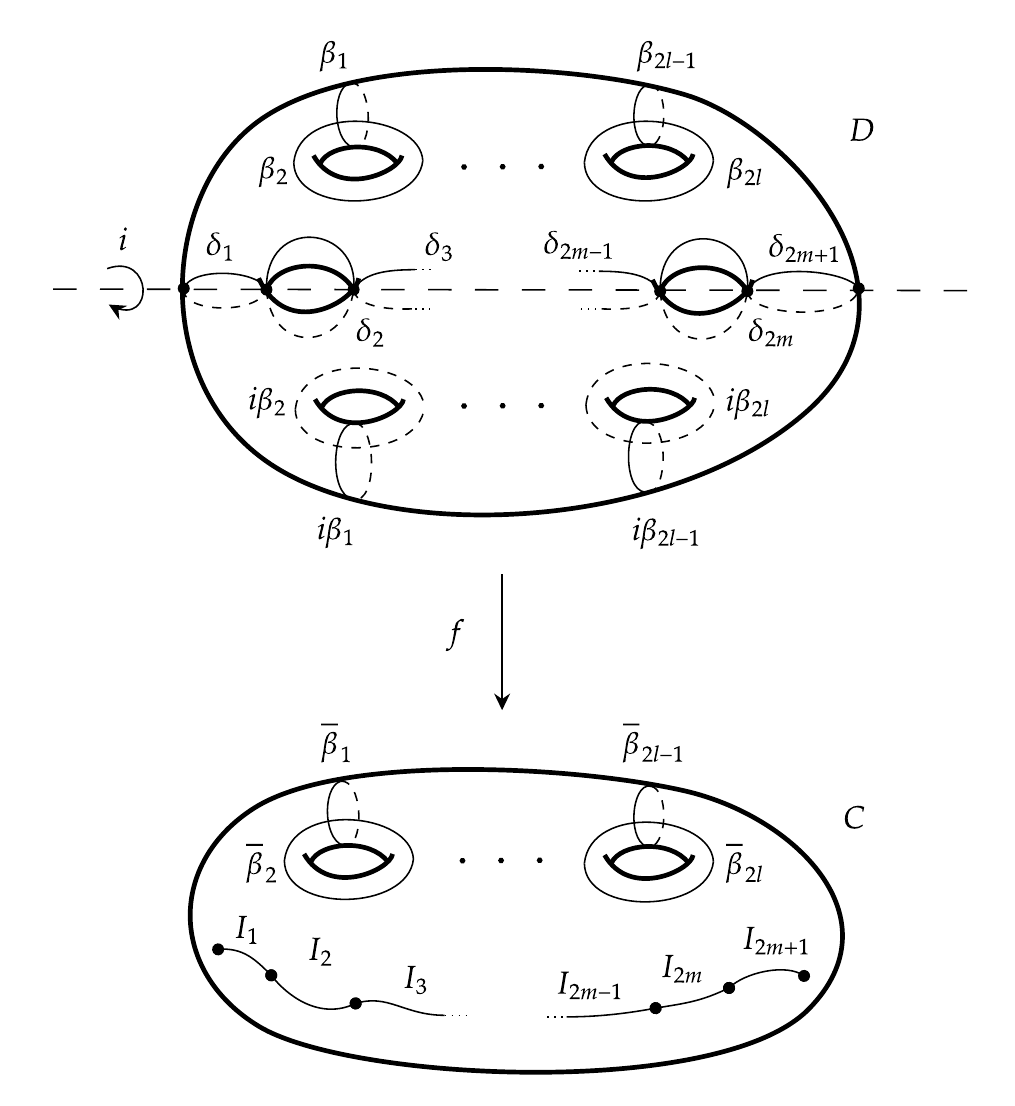}
  \caption{Double cover $f$}
  \label{Cover}
\end{figure}

It follows that the classes in $H_1(D,\mathbb{Z})$ of $\{\alpha_r,\alpha_s-i\alpha_s\mid r\in\mathfrak{r},\;s\in\mathfrak{s}\}$ lie in the anti-invariant homology $H_1(D,\mathbb{Z})_-$.
In fact, we will show that these classes generate the anti-invariant homology.

\begin{proposition} \label{anti_invariant_H_1}
    The group $H_1(D,\mathbb{Z})_-$ is generated by the classes of the cycles 
    \[\{\alpha_r,\alpha_s-i\alpha_s\mid r\in\mathfrak{r},\;s\in\mathfrak{s}\}.\]
\end{proposition}
\begin{proof}
    Topologically, the covering involution $i$ of $D$ is induced by a rotation in $\mathbb{R}^3$. More concretely, there is an embedding $\phi\colon D\hookrightarrow\mathbb{R}^3$ such that $\phi(D)$ is smooth, symmetric with respect to the $xy$- and $xz$-planes, and $i$ corresponds to the rotation by $180^{\circ}$ about the $x$-axis. A modern reference for this classical result is \cite[Theorem~5.7]{Dugger} (see also Remark \ref{trivial_cover}).

    In Figure~\ref{Cover}, an explicit picture of the involution is drawn.
    We will use the notation from this figure for the rest of the proof. The axis of rotation contains all the fixed points of $i$ and the $i$-anti-invariant cycles $\delta_1,\ldots,\delta_{2m+1}$ each contain two of the fixed points. The union $I$ of the embedded closed intervals $I_k:=f(\delta_k)$, $1\leq k \leq 2m+1$ in $C$ is again an embedded interval.
    Remark that the restriction $$f': D\setminus{\bigcup_{k=1}^{2m+1}\delta_k\rightarrow C\setminus I}$$ of $f$ is a trivial double cover and the preimage of $C\setminus I$ consists of two connected components $D_1$ and $D_2$. In particular, we have $\beta_j\subset D_1$ and $i\beta_j\subset D_2$ for $1\leq j\leq 2l$.
    
    The set of homology classes of cycles $\{\beta_j,i\beta_j, \delta_k \mid 1\leq j\leq2l, 1\leq k\leq 2m\}$ is a basis of $H_1(D,\mathbb{Z})$. Let $\bar{\beta}_j:=f_*\beta_j=f_*i\beta_j$. The homology classes of elements $\{\bar{\beta}_j \mid 1\leq j\leq2l\}$ form a basis of $H_1(C,\mathbb{Z})$.
    If $\bar{\alpha}_s$ is a vanishing cycle for $\gamma_s$, we can assume that it does not intersect $I$. Then $f^{-1}(\bar{\alpha}_s)$ consists of two connected components $\alpha_s\subset D_1$ and $i\alpha_s\subset D_2$ which are disjoint from the cycles $\delta_k\subset f^{-1}(I)$.
    
    In what follows, $\alpha_s$, $\beta_j$, and $\delta_k$ stand for the homology classes of the corresponding cycles. Additionally, when talking about linear spans of homology classes, it is implied that the the sum is taken over all relevant indices $1\leq j\leq2l$, $1\leq k\leq 2m$ and $s\in \mathfrak{s}$. 
    Note that $H_1(D,\mathbb{Z})$ is the orthogonal direct sum $\langle\delta_k\rangle\oplus\langle\beta_j\rangle\oplus\langle i\beta_j\rangle$,
    and that $\langle \alpha_s\rangle\subset\langle \delta_k,i\beta_j\rangle^{\perp}=\langle\beta_j\rangle$.
    We claim that the composition
    $$
    \langle \alpha_s \rangle \hookrightarrow \langle \beta_j \rangle \xrightarrow{f_*} H_1(C,\mathbb{Z})
    $$
    is surjective. Indeed, on the one hand, $f_*\langle \alpha_s, i\alpha_s, \alpha_r \rangle=f_*\langle \beta_j, i\beta_j, \delta_k \rangle=H_1(C,\mathbb{Z})$. On the other hand, the classes $\alpha_r$ are in the kernel of $f_*$ by Proposition \ref{invariant vanishing cycle} and $f_*\langle \alpha_s,i\alpha_s,\alpha_r \rangle=f_*\langle\alpha_s\rangle$.
    Since $f_*|_{\langle \beta_j \rangle}$ is an isomorphism, we now have $\langle \alpha_s \rangle = \langle \beta_j \rangle$.
    Note that the $i$-anti-invariant subspace $\langle \beta_j, i\beta_j \rangle_-$ coincides with $(1-i)\langle \beta_j \rangle$ since $\{\beta_j, i\beta_j\}$ is a basis of $\langle \beta_j, i\beta_j \rangle$. Hence we also have $\langle \alpha_s, i\alpha_s \rangle_-=(1-i)\langle \alpha_s \rangle$. 
    Now consider a class
    $$\alpha=\sum_ra_r[\alpha_r]+\sum_sb_s[\alpha_s]+\sum_s c_s[i\alpha_s]\in H_1(D,\mathbb{Z}).$$
    Assume that $\alpha$ is anti-invariant, which  holds if and only if the class $A=\sum b_s[\alpha_s]+\sum c_s[i\alpha_s]$ is anti-invariant.
    It follows that $A$ lies in $\langle \alpha_s, i\alpha_s \rangle_-=(1-i)\langle \alpha_s \rangle$. Therefore, $H_1(D,\mathbb{Z})_-$ is generated by $\{[\alpha_r],[\alpha_s]-[i\alpha_s]\mid r\in\mathfrak{r},s\in\mathfrak{s}\}$.
\end{proof}

\begin{remark} \label{trivial_cover}
    The result \cite[Theorem~5.7]{Dugger} is more general than required for the proposition above. 
    What we actually use is that, in our setting, the double cover $f$ is trivial over the complement of the union of the pairwise non-intersecting embedded closed intervals in $C$, each connecting a pair of ramification points.
    We outline an algebraic proof of this fact.
    A double cover $p(L,s)$ of $C$ branched at $2d$ points is determined by a line bundle $L$ of degree $d$ and a global section $s$ of $L^2$.  If $s=t^2$ with $t\in H^0(C,L)$, then the double cover corresponding to $s$ is trivial over the complement of the vanishing locus of $t$.
    We claim that the space $\mathcal{X}_{C,2d}$ parametrizing double covers of $C$ with $2d$ (possibly colliding) branch points is connected, so we can deform the cover $f$ to a double cover of the form $p(L,t^2)$.
    To see that $\mathcal{X}_{C,2d}$ is connected, we can argue as follows.
    The space $\mathcal{X}_{C,2d}$ is isomorphic to the fiber product of $\Sym^{2d}C$ and $\Pic^d(C)$ over $\Pic^{2d}(C)$, where $\Sym^{2d}C \to \Pic^{2d}(C)$ is the natural map and $\Pic^d(C) \to \Pic^{2d}(C)$ is multiplication by $2$. Moreover, by \cite[Remark~1]{Cornalba}, moving a ramification point of $p(L,s)$ along a simple closed loop $\gamma\in D$, results in the line bundle $L$ being tensored by the $2$-torsion element corresponding to $\gamma/2$ in $\Pic^0(C)$. Hence, $\mathcal{X}_{C,2d}$ is connected and we are done.
\end{remark}

\subsection{Monodromy action \& conclusion of proof}

Recall that we need to prove that the map 
$[\pi_1(P_0),\pi_1(U)]\to \pi_1(P_0)$ is surjective.
In other words, letting $c$ vary over $\pi_1(P_0)$ and $\gamma$ over our generators of $\pi_1(U)$, and denoting by $\widetilde{\gamma}$ the image $s_*\gamma$ of $\gamma$ under the section $s\colon U\to \cP_U$, we need to show that the elements $c^{-1}\widetilde{\gamma}^{-1}c\widetilde{\gamma}$ generate all of $\pi_1(P_0)$.
Note that the element $\widetilde{\gamma}^{-1}c\widetilde{\gamma}$ lies in $\pi_1(P_0)$ and is the result of the monodromy action of $\gamma\in \pi_1(U)$ on $c\in \pi_1(P_0)$ -- indeed, the monodromy action is given by 
\begin{align}\label{eq_monodromy}
    \PL\colon \pi_1(U)&\to\aut(\pi_1(P_0))\\
    [\gamma] &\mapsto \bigl\{c\mapsto \widetilde{\gamma}^{-1}c\widetilde{\gamma}\bigr\}. \nonumber
\end{align}
Using that $\pi_1(P_0)$ can be identified with $H_1(D,\mathbb{Z})_-\subset H_1(D,\mathbb{Z})=\pi_1(J(D))$, 
we identify the action \eqref{eq_monodromy} of $[\gamma]$ on $\pi_1(P_0)=H_1(D,\mathbb{Z})_-$ with the corresponding action of $g_*[\gamma]$ on $H_1(D,\mathbb{Z})$.
This action can be described using Picard--Lefschetz theory. We use the conventions of \cite[\S10.9]{ACG}.

\medskip
First, if $\gamma=\gamma_r$ is a loop around a branch of $W$ of type 1, then $g_*[\gamma_r]=[\lambda_r]$ where $\lambda_r$ can be viewed as the boundary of a small disk $K$ in $D$ intersecting $W^D_{\gamma_r}$ transversely in one general point. Thus, $\lambda_r$ corresponds to a family of smooth curves degenerating to a curve with one node. 
The action of $[\lambda_r]$ on $H_1(D,\mathbb{Z})$ is the action $T_{\alpha_r}$ induced by the Dehn twist along the vanishing cycle $\alpha_r$. This is given by
\[c\mapsto c+(c\cdot[\alpha_r])[\alpha_r].\]
Second, if $\gamma_s$ is a loop around a branch of $W$ of type 2, then $g_*[\gamma_s]=[\lambda_s][i\lambda_s]$ where similarly, $\lambda_s$ and $i\lambda_s$ correspond to families of smooth curves degenerating to nodal curves.
The action of $[\lambda_s][i\lambda_s]$ on $H_1(D,\mathbb{Z})$ is the composition $T_{i\alpha_s}\circ T_{\alpha_s}$. As $[i\alpha_s]\cdot[\alpha_s]=0$, this is given by 
$c\mapsto c+(c\cdot[\alpha_s])[\alpha_s] + (c\cdot [i\alpha_s])[i\alpha_s]$.
If $c$ is an anti-invariant class, then we have $(c.[i\alpha_s])=-(c.[\alpha_s])$, so we find (see also \cite[(7.13)]{ASF}) 
\[T_{i\alpha_s}\circ T_{\alpha_s}(c) = c+(c\cdot[\alpha_s])([\alpha_s] - [i\alpha_s]).\]

We now compute the images $c^{-1}\tilde{\gamma}^{-1}c\tilde{\gamma} \in\pi_1(P_0)$, or equivalently, 
$-c+\PL([\gamma])(c)\in H_1(D,\mathbb{Z})$,
for $c\in H_1(D,\mathbb{Z})_-$ and $[\gamma]$ ranging over our generators of $\pi_1(U)$. We find
\begin{align*}
 -c+\PL([\gamma_r])(c) &= (c\cdot[\alpha_r])[\alpha_r];\\
 -c+\PL([\gamma_s])(c) &= (c\cdot[\alpha_s])([\alpha_s ]- [i\alpha_s]).
\end{align*}
We want to show that the elements $-c+\PL([\gamma])(c)$ generate $H_1(D,\mathbb{Z})_-$. 
In fact, we can prove the following:

\begin{proposition}\label{prop_generatorsPL}
All generators $\{[\alpha_r],[\alpha_s]-[i\alpha_s]\mid r\in\mathfrak{r},s\in\mathfrak{s}\}$ of $H_1(D,\mathbb{Z})_-$ are of the form $-c+\PL([\gamma])(c)$ for some $c\in H_1(D,\mathbb{Z})_-$.
\end{proposition}

For the generators $[\alpha_s]-[i\alpha_s]$, this is proven as in \cite[Proof of Theorem~7.1]{ASF}: one finds an element $c\in H_1(D,\mathbb{Z})_-$ such that $c.[\alpha_s]=1$, and then $-c+\PL([\gamma_s])(c)$ equals $[\alpha_s]-[i\alpha_s]$.
For the generators $[\alpha_r]$, we use the following lemma:

\begin{lemma}\label{lemma_intersection1}
    If $C.B>2$, then for every $r\in\mathfrak{r}$, there exists $c\in H_1(D,\mathbb{Z})_-$ such that $c.[\alpha_r]=1$.
\end{lemma}

\begin{proof} 
    By Proposition \ref{invariant vanishing cycle}, we can assume that $\alpha_r$ is anti-invariant, non-separating, and smooth.
    We will construct an anti-invariant curve $\zeta_c$ which intersects $\alpha_r$ at exactly one point.
    
    Let $p$ be one of the two fixed points of $i$ on $\alpha_r$.
    Locally around $p\in D$, the involution $i$ looks like    $\mathbb{C}\ni z\mapsto -z$; the anti-invariant curves are the lines $\{z=x+\sqrt{\smash{-}1}y\mid y=0\}$ and $\{z=x+\sqrt{\smash{-}1}y\mid x=0\}$. 
    The cycle $\alpha_r$ must be tangent to one of these at $p$, say to $\{z=x+\sqrt{\smash{-}1}y\mid y=0\}$.  
    Since $C.B>2$, there exists a ramification point $q\in D$ not lying on $\alpha_r$.
    Now draw a path $\zeta\colon[0,1]\to D$ connecting $p$ and $q$, satisfying the following conditions:
    \begin{enumerate}
        \item on a small open neighbourhood of $p$, the image of $\zeta$ agrees with $\{z=x+\sqrt{\smash{-}1}y\mid x=0, y\geq 0\}$;
        \item the image of $(0,1]$ under $\zeta$ lies in $D\backslash\alpha_r$ (this is possible because $D\backslash\alpha_r$ is connected).
    \end{enumerate}
    Then we define the closed curve $\zeta_c\colon [0,1]\to D_0$ as follows: 
    \[\zeta_c(t) = \begin{cases}
    \zeta(2t)       & \quad \text{if } t\in[0,\frac12];\\
    i(\zeta(2(1-t))  & \quad \text{if } t\in[\frac12,1].
    \end{cases}
    \]
    This curve is anti-invariant by construction, and as $i$ preserves $D\backslash\alpha_r$, we see that the image of $(0,1)$ under $\zeta_c$ lies in $D\backslash\alpha_r$. Therefore $p$ is the only intersection point of $\zeta_c$ and $\alpha_r$. Locally around $p$, the cycle $c$ agrees with $\{z=x+\sqrt{\smash{-}1}y\mid x=0\}$, hence we find $[\zeta_c].[\alpha_r]=\pm 1$. Finally, we take $c=\pm[\zeta_c]$. 
\end{proof}

Now for $c$ as in Lemma~\ref{lemma_intersection1}, we have $-c+\PL([\gamma_r])(c)=[\alpha_r]$, which completes the proof of Proposition~\ref{prop_generatorsPL}.
We conclude that the map 
$[\pi_1(P_0),\pi_1(U)]\to \pi_1(P_0)$ is surjective.
This finishes the proof of Theorem~\ref{thm_pi1}.

\begin{remark}
    The statement above fails if $C.B=2$. As in the proof of Proposition \ref{anti_invariant_H_1}, one can view the involution $i$ as a rotation in $\mathbb{R}^3$, and thus choose a basis of homology of $D$ consisting of classes represented by anti-invariant simple closed curves $\delta_k$ and pairs of curves $\beta_j, i\beta_j$ that get switched by the rotation. Under our assumptions there are only two fixed points on $D$, so there is only one anti-invariant curve in the chosen basis which we denote by $\delta$. A vanishing cycle $\alpha_r$ is anti-invariant and its homology class is equal to $a\delta+\Sigma b_j(\beta_j-i\beta_j)$. Its intersection with any anti-invariant homology class $c$ is even since $c$ has the form $k\delta+\Sigma l_j(\beta_j-i\beta_j)$ and $\delta.\delta=0$.
\end{remark}

\section{Curves with non-integral preimage}
\label{section_curvesreduciblepreim}

In view of assumption (ii) of Theorem \ref{thm_pi1}, in this section we study the locus of curves in $|C|$ whose inverse image under $f$ is not integral. We show that under some numerical conditions on $|C|$, this locus has codimension at least 2. Together with the results of the next Section \ref{section_curvessingularpreim}, this will allow us in Section \ref{sec:examples} to give examples of linear systems for which the corresponding relative Prym varieties are irreducible symplectic varieties.

We recall the definition of $n$-connectedness:

\begin{definition}{\cite{Bombieri}} \label{n_connected}
    Let $n\in\Z_{>0}$. A linear system $|L|$ is called \emph{$n$-connected} if $C_1.C_2\geq n$ for all effective divisors $C_1+C_2\in |L|$.
\end{definition}

The main result, which is used in the proof of Theorem \ref{thm_Pisv}, is the following.

\begin{theorem}\label{thm_curvesreduciblepreim}
Let $C\subset T$ be a smooth curve with $C^2>0$ satisfying
 \begin{enumerate}[(a)]
 \item $C.B>0$ and $C.B_0>0$ for every rational irreducible component $B_0$ of $B$;
 \item $C^2$ and $C.B$ are not both equal to 4;
  \item $|C|$ is 2-connected;
  \item if $B^2\leq 0$, then $|C|$ is not hyperelliptic. 
  \end{enumerate}
 Then the locus of curves in $|C|$ whose inverse image under $f$ is not integral has codimension at least 2 in $|C|$.
\end{theorem}

We distinguish two subloci, which we treat separately: the locus of non-integral curves in $|C|$, and the locus of integral curves whose inverse image is non-integral. We give the proof of Theorem \ref{thm_curvesreduciblepreim} at the end of the current section.

\subsection{Integral curves with non-integral preimage}

Suppose $C_0$ is an integral curve on $T$.
If $f^{-1}(C_0)$ is not integral, it is either non-reduced or reducible.

\begin{remark}\label{remark_non_reduced}
If $f^{-1}(C_0)$ is non-reduced, $C_0$ must be contained in the branch locus $B$ of $f$, that is $C_0$ is an irreducible component of $B$. 
Since there are only finitely many of these, we find that if $\dim|C|\geq 2$, then the locus of integral curves in $C$ whose preimage under $f$ is non-reduced has codimension at least 2.
\end{remark}

The codimension of the locus of integral curves with reducible inverse image depends on the intersection numbers $C^2$ and $C.B$.

\begin{proposition} \label{prop_irredcurveswithredpullback}
 Let $C$ be a curve on $T$ with $C^2>0$ and $C.B>0$.
 Let $Z$ be the locus of integral curves in $|C|$ whose inverse image under $f$ is reducible, and assume $Z$ is non-empty.
 \begin{enumerate}[(i)]
 \item If $C^2$ and $C.B$ are both equal to 4, then $Z$ has codimension 1 in $|C|$; 
 \item otherwise, $Z$ has codimension at least 2.
 \end{enumerate}
\end{proposition}
\begin{proof}
 Assume $C_0\in Z$ is an integral curve in $|C|$ such that $f^*C_0$ is reducible. Then $f^*C_0$ is of the form
 \[f^*C_0=A_0+i^*A_0\]
 for some integral curve $A_0\subset S$ with $f_*A_0=C_0$ and $i^*A_0\neq A_0$ (as curves). As $\Pic(S)=\Pic(S)_+$
 by assumption,
 $A_0$ and $i^*A_0$ are linearly equivalent. It follows that $|2A_0|=|f^*C_0|=|f^*C|$.
 Note that
 \[A_0^2=\frac14(2A_0)^2=\frac14(f^*C)^2=\frac12C^2.\]
Moreover, the adjunction formula gives
 \[2g(C)-2=C.(C+K_T)=C^2-\frac12C.(-2K_T)=C^2-\frac12C.B.\]

Under the pullback map $f^*\colon |C| \to |2A|$, $Z$ can be identified with the image of the injective map $|A|\to |2A|$, $A'\mapsto A'+i^*A'$.
Using Corollary~\ref{cor_dimCforB.C>0}, the codimension of $Z$ is 
\[
\dim|C|-\dim |A|=\frac12C^2+\frac14C.B-\frac14C^2-1=\frac14C^2+\frac14C.B-1
\]
which is equal to $1$ if and only if $C^2=C.B=4$.
\end{proof}

\subsection{Non-integral curves}

In this subsection, $C\subset T$ is a smooth integral curve satisfying $C.B>0$.
We will find mild assumptions under which the set of non-integral members of $|C|$ has codimension at least two. The main results are summarized in Proposition \ref{prop_2connectedmerged}.

Our first goal is to find, in Corollary~\ref{cor_differencedimensions}, a numerical criterion \eqref{eq_ourgoal} for the locus of non-integral members of $|C|$ to have codimension at least two.
We need a technical lemma for a linear system $|C|$ on a rational surface $T$. 
    Let $\Sigma$ be an irreducible component of the locus of non-integral curves in $|C|$, and assume $\Sigma\neq\emptyset$. In this subsection, when we write $C_1+C_2 \in \Sigma$, we assume both $C_i$ to be effective.

    \begin{lemma}\label{lemma_dommap_twocurves}
    There are curves $C_1,C_2\subset T$ with $C_1+C_2\in\Sigma$ such that the map
    \begin{equation} 
    p_{C_1,C_2}\colon|C_1|\times|C_2|\to \Sigma,\;\;
    (C_1',C_2')\mapsto C_1'+C_2',
    \end{equation}
    which is finite onto its image, is dominant.
    Moreover, we may assume $C_1$ is integral.
\end{lemma}
\begin{proof}
    We consider the map $p_{C_1,C_2}$ for any $C_1+C_2\in\Sigma$ with $C_1$ integral.
    The union of the images of these $p_{C_1,C_2}$ is all of $\Sigma$. 
    As $\Pic(T)$ is countable, there are only countably many products $|C_1|\times|C_2|$, implying that for some choice of $C_1+C_2$, the map $p_{C_1,C_2}$ is dominant.
\end{proof}

\begin{corollary}\label{cor_differencedimensions}
    Let $|C|$ be a linear system on a rational surface $T$.
    The locus of non-integral curves in $|C|$ has codimension at least two if and only if the inequality
    \begin{equation} \label{eq_ourgoal}
        \dim|C|-\dim|C_1|-\dim|C_2|\geq 2    
    \end{equation}
    holds for every non-integral member $C_1+C_2\in|C|$.   
\end{corollary}

Hence we need to investigate when the inequality \eqref{eq_ourgoal} holds.

\medskip
If $C=C_1+C_2$ is a non-integral curve with $C.B>0$, then one of $C_1.B$ and $C_2.B$ must be positive. We will always assume that $C_1.B>0$. Then we have the following formula for $\dim|C|$:

\begin{lemma} \label{lemma_dimCunderassumptions}
Assume $C_1+C_2\in |C|$ with $C_1.B>0$. Then $|C|$ has dimension
\[\dim|C_1|+C.C_2-g(C_2)+1=\dim|C_1|+\dim|C_2|+C_1.C_2-h^0(C_2,\omega_T|_{C_2}).\]
If we also assume $C_2.B>0$, then $\dim |C|$ equals $\dim|C_1|+\dim|C_2|-C_1.C_2$.
\end{lemma}
\begin{proof}
For $i=1,2$, let $f_i\colon C_i\to C$ be the inclusion of $C_i$ into $C$. Then there is an exact sequence
    \[0\to\cO_T(C)|_C\to f_{1,*}\cO_T(C)|_{C_1}\oplus f_{2,*}\cO_T(C)|_{C_2}\to \cO_{C_1\cap C_2}\to 0.\]
 Taking the induced long exact sequence and using $H^1(C,\cO_T(C)|_C)=0$ (Lemma~\ref{lemma_dimlinearsystemT}), we find that 
 \[\dim|C|=h^0(C,\cO_T(C)|_C)= h^0(C_1,\cO_T(C)|_{C_1})+ h^0(C_2,\cO_T(C)|_{C_2})-C_1.C_2.\]
 Moreover, we have 
 $h^1(C_i,\cO_T(C)|_{C_i})=0$ for $i=1,2$, and hence $h^0(C_i,\cO_T(C)|_{C_i})=\chi(C_i,\cO_T(C)|_{C_i})$. 
 By Riemann--Roch theorem for embedded curves, this is $C.C_i+1-g(C_i)$.
For $i=1$ this equals $\dim|C_1|+C_1.C_2$ by Corollary~\ref{cor_dimCforB.C>0},
hence we find the first formula.
The second formula follows from Lemma~\ref{lemma_dimlinearsystemT}.
If also $C_2.B>0$, then applying Corollary~\ref{cor_dimCforB.C>0} to $C_2$ gives the final statement.
\end{proof}

As a corollary, if we assume that $C_2.B>0$ as well, then \eqref{eq_ourgoal} holds under the condition that
$C_1.C_2\geq 2$.
\begin{corollary} \label{lemma_analogA9} 
	Let $|C|$ be a linear system on $T$ with $C.B>0$. Assume that $C_1+C_2\in|C|$ with
	$C_1.B>0$, $C_2.B>0$ and $C_1.C_2\geq 2$.  
	Then $\dim|C|-\dim|C_1|-\dim|C_2|$ is at least two.
\end{corollary}

We will now focus on the case $C_2.B\leq 0$. 
We will first assume that $C_2$ is integral. We distinguish three cases: $C_2^2<0$, $C_2^2>0$ and  $C_2^2=0$.
In the first case, $C_2$ is a rational component of $B$:
\begin{lemma}\label{lemma_C22<0} 
    Let $C'\subset T$ be an integral curve with $C'.B\leq 0$ and $(C')^2<0$. Then $g(C')=0$ and $C'$ is an irreducible component of $B$.
\end{lemma}
\begin{proof}
If $C'$ is not an irreducible component of $B$, then we have $C'.B=0$ since $C'$ effective, hence $2g(C')=(C')^2+2\leq 1$ and therefore $g(C')=0$.
However, since $C'$ and $B$ are disjoint, the double cover $f^*C'\to C'\cong\mathbb{P}^1$ is unramified, a contradiction. 
This shows that $C'$ must be an irreducible component of $B$. Then $C'=f_*R'$ for an irreducible component $R'$ of the ramification locus, so $R'$ is isomorphic to $C'$ and $f^*C'=2R'$. Then we have\[g(C')=g(R')=(R')^2/2+1=(C')^2/4+1.\]
It follows from $(C')^2<0$ that $g(C')=0$.  \end{proof}

As a consequence,  \eqref{eq_ourgoal} holds whenever $C_2.B\leq 0$, $C_2^2<0$ and $C.C_2>0$.

\begin{corollary}\label{cor_ourgoalforC22<0+integral} 
    Let $C_1+C_2\in |C|$ such that $C_2$ is integral, $C_2.B\leq 0$, $C_2^2<0$ and $C.C_2>0$. 
    Then $\dim|C|-\dim|C_1|-\dim|C_2|$ is at least two.
\end{corollary}
\begin{proof}
By Lemma~\ref{lemma_C22<0}, $C_2$ is an irreducible component of $B$ which is rational.
The condition $C_2^2<0$ implies that $\dim|C_2|=h^0(C_2,\cO_T(C_2)|_{C_2})$ (Lemma~\ref{lemma_dimlinearsystemT}) is zero. Then the statement follows from Lemma~\ref{lemma_dimCunderassumptions} since $C.C_2>0$.
\end{proof}

In the case that $C_2$ is integral, $C_2.B\leq0$ and $C_2^2\geq 0$, one shows that $C_2.B=0$ and $B^2\leq 0$:

\begin{lemma} \label{lemma_C22>=0}
Let $C'\subset T$ be an integral curve with $C'.B\leq 0$ and $(C')^2 \geq 0$. Then $C'.B=0$. Moreover, we have the following possibilities:
\begin{enumerate}
    \item If $(C')^2 >0$, then $B^2 < 0$.
    \item If $(C')^2=0$, then $B^2 \leq 0$.
\end{enumerate}    
\end{lemma}
\begin{proof}
Recall that by Nikulin's theorem \ref{Nikulinstheorem}, the fixed locus of the involution $i$ is either a disjoint union $E_1 \sqcup E_2$ of two elliptic curves on the K3 surface $S$, or a disjoint union of the form
\[F \sqcup R_1 \sqcup \dots \sqcup R_k,\]
where $R_1, \dots, R_k$ are rational curves, $F$ is a curve of genus $g(F)$, and $k$ and $g(F)$ depend on the main invariant of the involution $i$. 
Assume towards a contradiction that $C'.B<0$. Then $C'$ is an irreducible component of $B$.
As in the proof of Lemma~\ref{lemma_C22<0}, then  $C'=f_*R'$ for an irreducible component $R'$ of the ramification locus, satisfying $g(R')=(C')^2/4+1$.
It follows from $(C')^2\geq 0$ that $R'$ cannot be rational.
In particular, with the above notation, either $R'=F$ with $g(F) \geq 1$, or $R'$ is one of the components of $E_1\sqcup E_2$, say $R'=E_2$.
In the first case, we have
\[0 > C'.B=C'.f_*R=f^*C'.R=2 R'.(R' + R_1 + \dots + R_k)=2 (R')^2=(C')^2 \geq 0,\]
giving a contradiction. In the second case, similarly we get
\[0 > C'.B=2(R')^2 + 2R'.E_1 = 0,\]
which is not possible. We conclude that $C'.B=0$.

For the second part, if $(C')^2>0$, it follows from Hodge Index Theorem that $B^2<0$.
If $(C')^2=0$, then either $B^2 <0$ or, by \cite[Lemma 2.1]{KnLopK3}, $B$ and $C'$ are linearly equivalent to multiples of a primitive divisor $A$ with $A^2=0$.
\end{proof}

We can use the above to show that if $C_2^2>0$, then $\omega_T|_{C_2}$ is non-trivial.

\begin{lemma}\label{lemma_res_omega_nontrivial}
    Let $C'\subset T$ be an integral curve with $C'.B\leq 0$ and $(C')^2 > 0$.
    Then $\omega_T|_{C'}$ is non-trivial.
\end{lemma}
\begin{proof}
By Lemma~\ref{lemma_C22>=0}, we have $C'.B=0$.
Note that if $C'$ is an irreducible component of $B$, then $C'$ is disjoint from all other components, and $C'.B=0$ implies $(C')^2=0$.
So $C'$ is not an irreducible component of $B$.
Now if $\omega_T|_{C'}$ is trivial, then the étale double cover $D':=f^*C'\to C'$ is trivial (see also Section~\ref{subsubsec_Prymfib}), so $D'$ consists of two disjoint curves $D'_1$ and $D'_2$, both isomorphic to $C'$. 
If $(C')^2>0$, then both $D'_1$ and $D'_2$ have positive square, which is not possible by Hodge index theorem since they are disjoint.
\end{proof}

\begin{corollary}\label{cor_C22>0}
Let $C_1+C_2\in|C|$ with $C_1.B>0$ and $C_2$ integral with $C_2.B\leq 0$, $C_2^2>0$ and $C_1.C_2\geq 2$. Then $\dim|C|-\dim|C_1|-\dim|C_2|$ is at least two.
\end{corollary}
\begin{proof}
    By Lemma~\ref{lemma_C22>=0}, $\omega_T|_{C_2}$ is a line bundle of degree 0, which is non-trivial by Lemma~\ref{lemma_res_omega_nontrivial}.
    It follows that $h^0(C_2,\omega_T|_{C_2})=0$, and the claim follows from Lemma~\ref{lemma_dimCunderassumptions}.
\end{proof}

Finally, let $C'\subset T$ be an integral curve with $C'.B\leq 0$ and $(C')^2=0$, so in fact $C'.B=0$. Then we have $2g(C')-2=0$, so $g(C')=1$. If $\omega_T|_{C'}$ is trivial, then $\dim|C'|=1$ by Lemma~\ref{lemma_dimlinearsystemT}. 
Thus, $|C'|$ induces a genus-1 fibration $T\rightarrow\mathbb{P}^1$.
Then any smooth curve intersecting $C'$, and therefore all fibres, in 2 points, is a double cover of $\mathbb{P}^1$ and hence is hyperelliptic.

\begin{corollary}\label{cor_C22=0}
Suppose that $C$ is non-hyperelliptic.
Let $C_1+C_2\in|C|$ with $C_1.B>0$ and $C_2$ integral with $C_2.B\leq 0$, $C_2^2=0$ and $C_1.C_2\geq 2$.
Then $\dim|C|-\dim|C_1|-\dim|C_2|$ is at least two.
\end{corollary}
\begin{proof}
By Lemma~\ref{lemma_C22>=0}, $\omega_T|_{C_2}$ is a line bundle of degree 0, and $C_2$ has genus 1.
If $\omega_T|_{C_2}$ is non-trivial, then $h^0(C_2,\omega_T|_{C_2})$ is zero so by Lemma~\ref{lemma_dimCunderassumptions}, we have $\dim|C|-\dim|C_1|-\dim|C_2|=C_1.C_2\geq 2$.
If $\omega_T|_{C_2}$ is trivial, then we have
$\dim|C|-\dim|C_1|-\dim|C_2|=C_1.C_2-1$.
Now as explained above, if $C_1.C_2=C.C_2$ equals 2, then $C$ is hyperelliptic, which we assumed is not the case. Hence, we must have $C_1.C_2\geq 3$, and therefore
$\dim|C|-\dim|C_1|-\dim|C_2|\geq 2$.
\end{proof}

For general, not necessarily integral $C_2$ with $C_2.B\leq 0$, we will restrict to the integral case using the following consequence of Lemma~\ref{lemma_dommap_twocurves}:

\begin{corollary} \label{cor_reduciblecase}
Let $C_1+C_2\in|C|$ with $C_1.B>0$ and $C_2.B\leq 0$. Then we can find curves $C_1'$ and $C_2'$ satisfying
\begin{enumerate}[i)]
\item $C_1'+C_2'\in |C|$, $C_1'.B>0$ and $C_2'.B\leq 0$;
\item $C_2'$ is integral;
\item $\dim|C_1|+\dim|C_2|\leq \dim|C_1'|+\dim|C_2'|$.
\end{enumerate}
\end{corollary}
\begin{proof}
If $|C_2|$ contains an integral curve $C_2'$ then $C_1'=C_1$ and $C_2'$ satisfy all the required properties. 
If $|C_2|$ does not contain integral curves, then by Lemma~\ref{lemma_dommap_twocurves}, using induction, there exist integral curves $C_{2,1},\dots,C_{2,n}$ such that $C_{2,1}+\dots+C_{2,n}\in|C_2|$ and $\dim|C_2|=\dim\left(|C_{2,1}|\times\ldots\times|C_{2,n}|\right)$.
At least one of the curves $C_{2,i}$ has the property $C_{2,i}.B\leq 0$, say this curve is $C_{2,n}$. Set $C_2'=C_{2,n}$ and $C_1'=C_1+C_{2,1}+\ldots+C_{2,n-1}$. Then 
we have $C_1'.B=(C-C_{2,n}).B>0$ and
\begin{align*}
\dim|C_1|+\dim|C_2|&=\dim|C_1|+\dim|C_{2,1}|+\ldots+\dim|C_{2,n}|\\
&\leq\dim|C_1+C_{2,1}+\ldots+C_{2,n-1}|+\dim|C_{2,n}|\\
&=\dim|C_1'|+\dim|C_2'|. \qedhere
\end{align*}
\end{proof}

\begin{proposition}\label{prop_2connectedmerged}
Assume that $C$ intersects all rational components of $B$ positively, that $|C|$ is 2-connected, and if $B^2\leq 0$, that additionally $|C|$ is not hyperelliptic. 
Then the set of non-integral 
members in $|C|$ has codimension at least two.
\end{proposition}
\begin{proof}
Let $C_1+C_2\in|C|$, with $C_1.B>0$. We need to show that \eqref{eq_ourgoal} holds. 
If $C_2.B>0$, this follows from Corollary~\ref{lemma_analogA9}. If $C_2.B\leq 0$, we may assume by Corollary~\ref{cor_reduciblecase} that $C_2$ is integral.
If $C_2^2$ is nonnegative, then by Lemma~\ref{lemma_C22>=0} we have $B^2\leq 0$, so by assumption, $|C|$ is not hyperelliptic. Then \eqref{eq_ourgoal} holds by Corollaries~\ref{cor_C22>0} and \ref{cor_C22=0}.
If $C_2^2$ is negative, then by Lemma~\ref{lemma_C22<0}, $C_2$ is a rational component of $B$, hence we have $C.C_2>0$ by assumption. Then \eqref{eq_ourgoal} follows from Corollary~\ref{cor_ourgoalforC22<0+integral}. 
\end{proof}

\begin{proof}[Proof of Theorem~ \ref{thm_curvesreduciblepreim}]
We want to prove that the locus of curves in $|C|$ whose inverse image under $f$ is not integral has codimension at least 2. 
This locus is a union of three subloci: the locus of integral curves whose inverse image is non-reduced, the locus of integral curves whose inverse image is reducible, and the locus of non-integral curves.
Under the assumptions (a)--(d), these three loci all have codimension at least 2 in $|C|$, by Remark~\ref{remark_non_reduced}, Proposition~\ref{prop_irredcurveswithredpullback} and Proposition~\ref{prop_2connectedmerged}, respectively.
\end{proof}

\section{Curves with singular preimage}
\label{section_curvessingularpreim}

In view of assumption (iii) of Theorem \ref{thm_pi1}, in this section, we study the locus of curves in $|C|$ whose pullback by $f$ is singular. The goal is to describe the general element of each codimension $1$ irreducible component of this locus.
We will prove the following:

\begin{theorem} \label{thm_codim1loci}
 Let $C\subset T$ be a very ample effective divisor. 
Assume that the locus of non-integral curves in $|C|$ has codimension at least 2.
 Let $Z\subset |C|$ be an irreducible component of the locus of curves in $|C|$ whose inverse image under $f$ is singular. If $Z$ has codimension 1 in $|C|$, then one of the following holds:
 \begin{enumerate}[(i)]
  \item The general element of $Z$ is an integral curve with one node and no other singularities, intersecting the branch locus $B$ transversely;
  \item The general element of $Z$ is a smooth curve intersecting $B$ transversely except in one point, where the intersection multiplicity is 2.
 \end{enumerate}
\end{theorem}

The proof consists of Proposition~\ref{prop_singularnontransvcurves}, which concerns the discriminant locus $\Delta$ of singular curves,
and Proposition~\ref{prop_1pointmult2}, which is about the locus $\Delta'$ of curves intersecting the branch locus $B$ non-transversely.
Indeed, by Lemma~\ref{lemma_milnornumber}, all other curves have smooth inverse image.

\medskip
Our first goal is to show that a general curve in the discriminant locus $\Delta$ intersects $B$ transversely.
We will do this by showing that the intersection $\Delta\cap\Delta'$ has codimension at least 2 in $|C|$.
The general element of $\Delta$ 
is an integral curve with one node and no other singularities \cite[Corollary~2.8]{VoisinII}.
Recall that $B$ is smooth, so a curve $C$ intersects $B$ non-transversely at a point $p$ if and only if the tangent space $T_pB$ is contained in $T_pC$.

\medskip
Let $N=\dim|C|$ and consider the embedding 
$\varphi_C\colon T\hookrightarrow|C|^{\vee}=\mathbb{P}^N$.
Denote by $\overline{T}$ the image of $T$ under $\varphi_{C}$, and similarly, let $\overline{p}=\varphi_C(p)$ for any point $p\in T$.
Singular curves in $|C|$ are parametrized by the dual $\overline{T}^{\vee}\subset\check{\mathbb{P}}^{N}$ of $T$ (see e.g.\ \cite[\S1.1]{GeKaZe}).
More precisely, for a hyperplane $H\subset|C|^{\vee}$, $H\cap \overline{T}$ is singular at a point $\overline{p}\in\overline{T}$ if and only if 
$T_{\overline{p}}\overline{T}\subset H$.

Let $B_0$ be an irreducible component of $B$, with image $\overline{B_0}$ under $\varphi_C$. 
The locus of curves in $|C|$ that intersect $B_0$ non-transversely is parametrized by the dual variety 
$\overline{B_0}^{\vee}\subset|C|^{\vee\vee}\cong|C|$. 
Again, more precisely, 
the curve $H\cap\overline{T}$ intersects $\overline{B_0}$ non-transversely at $\overline{p}$ if and only if $T_{\overline{p}}\overline{B_0}\subset H$.

\begin{proposition} \label{prop_singularnontransvcurves}
Let $|C|$ be a very ample linear system on $T$ such that the locus of non-integral curves in $|C|$ has codimension at least 2. 
Let $B_0\subset T$ be an irreducible component of the branch locus $B$. The locus of singular curves in $|C|$ that intersect $B_0$ non-transversely has codimension at least 2 in $|C|$.
\end{proposition}
\begin{proof}
Since $\overline{B_0}^{\vee}\subset|C|^{\vee\vee}\cong|C|$ is irreducible of codimension at least 1, it suffices to show that either $\overline{B_0}^{\vee}$ has codimension at least 2, or $\overline{B_0}^{\vee}$ is not contained in $\Delta$, that is, there is a curve in $|C|$ intersecting $B_0$ non-transversely which is smooth.
We distinguish two cases.

First, suppose there is a point $p\in B_0$ such that $T_{\overline{p}}\overline{B_0}\not\subset \overline{T}$. 
Note that when $N:=\dim|C|$ equals 2, so $\varphi_{C}$ is an isomorphism, this cannot happen, so we have $N\geq 3$.
We will show that the subsystem $|C|_1\subset|C|$ of curves intersecting $B_0$ non-transversely at $p$ contains a smooth curve.
Then $|C|_1$ corresponds to the subsystem $|H|_1\subset|H|$ of hyperplanes containing $T_{\overline{p}}\overline{B_0}$, which has dimension $N-2$. 
By Bertini's theorem, there is a dense open subset $\cU\subset |C|_1\cong \mathbb{P}^{N-2}$ such that every element of $\cU$ is smooth away from the base locus. 
The base locus of $|H|_1$ is the closure $\mathbf{T}_p$ of $T_{\overline{p}}\overline{B_0}$ in $|C|^{\vee}$; hence the base locus of $|C|_1$ is $\mathbf{T}_p\cap\overline{T}$. 
Since $T_{\overline{p}}\overline{B_0}\not \subset \overline{T}$, it follows that the base locus of $|C|_1$ consists of finitely many points $p_1:=p,p_2\dots,p_k$.
The locus of curves in $|C|$ which are singular at $p_i$ corresponds to $P_i:=\{H\in |H|\mid T_{\overline{p_i}}\overline{T}\subset H\}\cong\mathbb{P}^{N-3}$.
So the locus of curves in $|C|$ that intersect $B_0$ non-transversely at $p$ and are smooth at each $p_i$ is $|C|_1\setminus\bigcup_i\left(P_i\cap|C|_1\right)$, which is a dense open subset of $|C|_1$. Its intersection with $\cU\subset|C|_1$ consists of smooth curves that intersect $B_0$ non-transversely at $p$, and this intersection is non-empty.

If $T_{\overline{p}}\overline{B_0}\subset \overline{T}$ for all $p\in B_0$,
then for any $H\in\overline{B_0}^{\vee}$, the curve $H\cap\overline{T}$ contains $T_{\overline{p}}\overline{B_0}$ for some $p$.
Then the closure $\mathbf{T}_p$ of $T_{\overline{p}}\overline{B_0}$ in $|C|^{\vee}=
\check{\mathbb{P}}^{N}$ is an irreducible component of $H\cap\overline{T}$.
If $\mathbf{T}_p=H\cap\overline{T}$ for some $p$, then its pullback to $T$ is a smooth curve intersecting $B_0$ non-transversely. 
If $\mathbf{T}_p$ is always strictly contained in $H\cap\overline{T}$, then for all $H\in\overline{B_0}^{\vee}$, the curve $H\cap\overline{T}$ is reducible.
So the curves in $|C|$ intersecting $B_0$ non-transversely are contained in the locus of reducible curves, which has codimension at least 2 by assumption.
\end{proof}

Next, we will show that if $C_0$ is a general element of the locus $\Delta'\subset|C|$ intersecting $B$ non-transversely,
then $C_0$ has exactly one point of non-transverse intersection with $B$, of multiplicity 2.
By Proposition~\ref{prop_singularnontransvcurves},
we may assume that $C_0$ is smooth.

\begin{proposition}\label{prop_1pointmult2}
Let $|C|$ be a very ample linear system on $T$ such that the locus of non-integral curves in $|C|$ has codimension at least 2.
The locus of curves in $|C|$ which intersect $B$ non-transversely in more than one point or in a point with multiplicity bigger than two, has codimension at least 2 in $|C|$.
\end{proposition}
\begin{proof}
    Let $B_1,\dots,B_k$ be the irreducible components of $B$. 
    The locus $F_i\subset\overline{B_i}^{\vee}$ of curves in $|C|$ intersecting $B_i$ in either more than one point or with multiplicity bigger than 2 has codimension at least 2 in $|C|$
    \cite[Corollaire~3.2.1, Théorème~2.5]{Katz}.
    Hence, so does the locus $F_1\cup\dots\cup F_k$ of curves intersecting $B$ non-transversely in a point of multiplicity more than two, or in two points in the same component of $B$.
    We are left to consider the locus of curves intersecting $B$ non-transversely in two points in different components $B_i$ and $B_j$ of $B$.
    This locus corresponds to the intersection $\overline{B_i}^{\vee}\cap\overline{B_j}^{\vee}$, which has codimension at least 2 in $|C|$ unless one of the two varieties is contained in the other, say $\overline{B_i}^{\vee}\subset\overline{B_j}^{\vee}$.
    But in that case, it follows from biduality \cite[Theorem~1.1]{GeKaZe} that under the identification $|C|^{\vee\vee\vee}\cong|C|^{\vee}$, we have
    $B_i=\overline{B_i}^{\vee\vee}\subset\overline{B_j}^{\vee\vee}=B_j$, a contradiction.
    So $\overline{B_i}^{\vee}\cap\overline{B_j}^{\vee}$ has codimension at least 2 in $|C|$, concluding the proof.
\end{proof}

\begin{remark}
Proposition~\ref{prop_1pointmult2} can also be proven under the assumption that $|C|$ is only ample, and $|f^*C|$ is non-hyperelliptic. We did not add this much lengthier proof, since our main results require $|C|$ to be very ample.
\end{remark}

\section{Examples}\label{sec:examples}

Using Theorem~\ref{thm_Pisv}, whose hypotheses are satisfied when the linear system $|C|$ is positive enough, 
one can find infinitely many examples of irreducible symplectic varieties. In this section, we study some explicit examples of relatively low dimension. 
To be precise, we exhibit examples of rational surfaces which are quotients of very general K3 surfaces with an anti-symplectic involution, together with a divisor satisfying the hypotheses of Theorem~\ref{thm_Pisv}. 
Since the case where $T$ is an Enriques surface has already been covered by \cite{ASF}, we focus on cases where the double cover $S\to T$ is ramified.

We start in Section \ref{sec_projplane} by taking for $T$ the projective plane and by letting $C$ be a multiple of a line. This gives infinitely many examples which
satisfy the conditions of Theorem~\ref{thm_Pisv} and thus yield irreducible symplectic varieties.
Next, we take for $T$ a del Pezzo surface.
In this case, there are already some examples in the literature. We revise these examples and give a generalization of one of them in Corollary~\ref{cor_Matteini_generalized}, obtaining irreducible symplectic varieties of dimensions \(6,8,10,12,14,16\).
We also find new examples with a linear system that has not been studied before:
in Proposition~\ref{prop_example del Pezzo} we consider $T$ a del Pezzo surface of degree $1 \leq d \leq 8$; for a certain choice of the linear system $|C|$ we obtain 
examples of irreducible symplectic varieties in arbitrarily high dimension, starting from dimension 8. 
Finally, in Proposition~\ref{prop_ex_non-primitive linear system} we see that in one of the known examples, where the relative Prym variety is known to be symplectic, our results imply that it is actually irreducible symplectic.

\medskip
The next lemma shows in which cases the ramification divisor \(B\) has positive square, allowing us to understand in which cases we need to check if \(C\) is hyperelliptic in relation with condition \((5)\) of Theorem~\ref{thm_Pisv}.
\begin{lemma}\label{lemma_sign_B}
	Let $(r,a,\delta)$ be the main invariant of $(S,i)$. Then $B^2$ equals $4(10-r)$. In particular, we have $B^2>0$ if and only if $r<10$.
\end{lemma}

\begin{proof}
	If the main invariant of $(S,i)$ equals $(10,10,0)$ or $(10,8,0)$ then by Nikulin's Theorem \ref{Nikulinstheorem}, $B$ is empty or a disjoint union of two elliptic curves, respectively. In both cases we have $B^2=0$. 
 For all the other possible main invariants, $B$ is a union of a curve of genus $g=11-\frac{r+a}{2}$ and $k=\frac{r-a}{2}$ rational curves. As in the proof of Lemma \ref{lemma_C22<0}, for each component $B_j$ of $B$ we have $g(B_j)=B_j^2/4+1$. It follows that $B^2=\Sigma B_j^2=4(g-1-k)=4(10-r)$.
\end{proof}

We will use the following numerical criterion for 2-connectedness of an ample divisor on a smooth
surface.

\begin{theorem}{\cite[Theorem A]{BL}}\label{thm_BeltraLanteri}
Let $X$ be a smooth surface and $C$ an ample divisor on $X$. Then $C$ is $2$-connected, unless:
\begin{enumerate}
    \item[A1)] $C^2=2$, $X$ is a quadric surface in $\mathbb{P}^3$ and $C$ is the restriction of the hyperplane class on $\mathbb{P}^3$. 
    \item[A2)] $C^2=4$, $X=\mathbb{P}^2$, $C=\mathcal{O}(2)$.
    \item[A3)] $C^2=4$, $C=L_1 \otimes L_2$, with $L_1$ numerically equivalent to $L_2$, and the polarized pair $(X, L_i)$ has $\Delta$-genus $1$ or $2$ for $i=1,2$.
    \item[A4)] $X$ is a $\mathbb{P}^1$-bundle and $C$ is a section plus some fibers.
\end{enumerate}
\end{theorem}

In the case of a K3 surface, we can use Reider's theorem to obtain a numerical characterization of very ampleness.

\begin{theorem}{\cite[Theorem 1 and Remark 1.2,2)]{Reider}}\label{thm_Reider}
Let $D$ be a nef divisor on a smooth projective surface $X$. If $D^2>8$, then $K_X+D$ is very ample unless there exists a nonzero effective divisor $E$ satisfying one of the following: 
\begin{enumerate}
    \item $D.E=0$, $E^2=-1, -2$;
    \item $D.E=1$, $E^2=0, -1$;
    \item $D.E=2$, $E^2=0$;
    \item $D.E=3$, $D=3E$, $E^2=1$.
\end{enumerate}
\end{theorem}

\subsection{Projective plane}\label{sec_projplane}

We consider a general K3 surface $S$ which is a double cover of $\mathbb{P}^2$ branched in a smooth sextic curve $B\subset\mathbb{P}^2$.
In Nikulin's classification, this is the case $(1,1,1)$.
Let $C_0\subset \mathbb{P}^2$ be a line, with inverse image $D_0\subset S$.
We consider the linear system $|nC_0|=|\cO_{\mathbb{P}^2}(n)|$.
Its genus is $\frac12(n-1)(n-2)$; 
the genus of $f^*|nC_0|=|nD_0|$ is $n^2+1$.
The linear system $|D|=|nD_0|$ is very ample for $n\geq 3$ (see e.g.\ \cite[Theorem~2.2.7]{Huy_bookK3});  moreover, $|C|$ is  2-connected by Theorem~\ref{thm_BeltraLanteri}.
Hypotheses \eqref{hyp_int>2} and \eqref{hyp_int_not4} of Theorem~\ref{thm_Pisv} are clearly satisfied as well, and \eqref{hyp_hyperell} holds because of Lemma~\ref{lemma_sign_B}.
Hence, we obtain:

\begin{proposition}
\label{prop_ExProjPlane}
Let $(S,i)$ be a very general K3 surface with anti-symplectic involution of type $(1,1,1)$, 
so that the quotient surface $T$ is $\mathbb{P}^2$. Let $|C|=|\cO_{\mathbb{P}^2}(n)|$ with $n\geq 3$.
Then the associated relative Prym variety $\cP_D$ is an irreducible symplectic variety, of dimension $n^2+3n$.
\end{proposition}

Note that in the case $n=2$, the linear system $|2D_0|$ is hyperelliptic with $i^*=j$ and the Prym variety equals $M_{2D_0}$ -- see Remark~\ref{remark_hyperell}.

\subsection{Del Pezzo surfaces}\label{Del Pezzo surfaces}

Let $(S,i)$ be a very general K3 surface with anti-symplectic involution of type $(10-d,10-d,1)$, so that the quotient $T$ is a del Pezzo surface of degree $d$ which is not $\p^1 \times \p^1$.
Since $T$ is the blow up of $\mathbb{P}^2$ in $9-d$ points, its Picard group is generated by $H, E_1, \dots, E_{9-d}$, where $H$ is the pullback of the hyperplane class on $\mathbb{P}^2$ and the $E_i$ are the exceptional divisors. A divisor $C$ on $T$ can be written as $C=aH - \sum_{i=1}^{9-d} b_iE_i$, for some integers $a$, $b_i$.
In particular, we have $H^2=1$, $H.E_i=0$, and $E_i.E_j=-\delta_{ij}$. 
Note that $K_T=-3H+E_1+ \dots E_{9-d}$, so the branch divisor $B$ is equivalent to $-2K_T=6H-2E_1- \dots -2E_{9-d}$. Its square $B^2=36-4(9-d)=4d$ is positive, as expected by Lemma~\ref{lemma_sign_B}. 
Hence, condition (5) of Theorem~\ref{thm_isv} is automatically satisfied.

Moreover the divisor $C$ satisfies $C^2=a^2-b_1^2-\ldots-b_{9-d}^2$, $C.K_T=-3a+b_1+\ldots+b_{9-d}$ and, using the genus formula,
\[g(C)=1+\frac12\Bigl(a^2-3a+\sum_{i=1}^{9-d}(b_i-b_i^2)\Bigr).\]
We recall the following numerical criterion for very ampleness on a del Pezzo surface. 

\begin{theorem}[\cite{dR}]\label{thm_diRocco}
Let $T$ be a del Pezzo surface of degree $d$. Consider a divisor $C=aH - \sum_1^{9-d} b_iE_i$ such that $C \neq -K_T, -2K_T$ if $d=1$, $C \neq -K_T$ if $d=2$. Then $C$ is very ample if and only if:
\begin{itemize}[leftmargin=12pt]
\item for $d=8$: $a \geq b_1+1$ and $b_1 \geq 1$;
\item for $d=7, 6, 5$: $a \geq b_i + b_j +1$ where $i \neq j=1, \dots, 9-d$, and $b_1 \geq b_2 \geq \dots \geq b_{9-d} \geq 1$;
\item for $d=4,3$: $a \geq b_i + b_j +1$ where $i \neq j=1, \dots, 9-d$, and $b_1 \geq b_2 \geq \dots \geq b_{9-d} \geq 1$, and $2a \geq \sum_{1}^5 b_{i_t}+1$;
\item for $d=2$: $a \geq b_i + b_j +1$ where $i \neq j=1, \dots, 9-d$, and $b_1 \geq b_2 \geq \dots \geq b_{9-d} \geq 1$, and $2a \geq \sum_{1}^5 b_{i_t}+1$, and $3a \geq 2b_i+ \sum_{1}^6 b_{j_t}+1$;
\item for $d=1$: $a \geq b_i + b_j +1$ where $i \neq j=1, \dots, 9-d$, and $b_1 \geq b_2 \geq \dots \geq b_{9-d} \geq 1$, and $2a \geq \sum_{1}^5 b_{i_t}+1$, and $3a \geq 2b_i+ \sum_{1}^6 b_{j_t}+1$, and $4a \geq \sum_1^3 2b_{i_t}+\sum_1^5 b_{j_t}+1$, and $5a \geq \sum_1^6 2b_{i_t}+b_j+b_k+1$, and $6a \geq 3b_i+ \sum_1^7 2b_{j_t}+1$.
\end{itemize}
\end{theorem}

Using the above results, checking the hypotheses of Theorem \ref{thm_isv} reduces to a numerical computation which allows us to produce many examples of Prym varieties.
Let us first resume the known examples.

\begin{example}\label{Ex_known}
\hfill
\begin{enumerate}
\item\label{Ex_Matteini_deg1} 
\cite[\S4.2]{Matteini}
Let $(S,i)$ be of type $(9,9,1)$, so that the quotient surface $T$ is del Pezzo of degree $1$, and let $C=-K_T$. In this case, $\mathcal{P}_D$ is an elliptic K3 surface.  
\item\label{Ex_MarkTikh} 
\cite{MT}
Let $(S,i)$ be of type $(8,8,1)$, so that the quotient surface $T$ is del Pezzo of degree $2$, and let $C=-K_T$. In this case, $\mathcal{P}_D$ is a singular symplectic variety of dimension $4$ without symplectic resolution -- see also \cite[\S4.3]{Matteini}.
Moreover we note that $\mathcal{P}_D$ is an irreducible symplectic orbifold of Nikulin type, as shown in \cite[Proposition 3.12]{MenetRiess}, and as such an irreducible symplectic variety by \cite[Proposition 3 (2)]{Perego}.
\item
\label{Example_Matteini}\cite[\S4.4]{Matteini}
Let $(S,i)$ be of type $(7,7,1)$, so that the quotient surface $T$ is del Pezzo of degree $3$, and let $C=-K_T$.
Then $\cP_{D}$ is a simply connected symplectic 6-fold with $h^{2,0}(\cP_D)=1$ and without symplectic resolution.
In fact, Theorem~\ref{thm_Pisv} implies that it is an irreducible symplectic variety, see Corollary~\ref{cor_Matteini_generalized} below.
\item\label{Ex_SS} 
\cite[\S3.2]{SS} 
Let $(S,i)$ be of type $(10-d,10-d,1)$ with $d=1,\dots,9$, so that the quotient surface $T$ is del Pezzo of degree $d$, and let $C=-2nK_T$ for any $n\geq 1$. 
Then $\cP_D$ is a symplectic variety of dimension $2n(2n + 1)d$ without symplectic resolution (see the last remark of \cite[\S3.2]{SS}).
In the case $n=d=2$, $\cP_D$ is birational to a quotient of a smooth simply connected projective
variety by an involution \cite[\S6.4]{ShenThesis}.
Hence, it follows from Proposition~\ref{prop_condition_for_psv} that $\cP_D$ is a primitive symplectic variety.
\end{enumerate}
\end{example}

In Example~\eqref{Ex_Matteini_deg1}, Theorem~\ref{thm_Pisv} does not apply since $C.B=2$. In \eqref{Ex_MarkTikh}, Theorem~\ref{thm_Pisv} cannot be applied either, as $C$ is not very ample.
In Example~\eqref{Example_Matteini} however, Theorem~\ref{thm_Pisv} does apply.
Indeed, we have $C^2=3 \neq 4$ and $C.B=2K_T^2=6$. 
Moreover, $C$ is very ample by Theorem~\ref{thm_diRocco} and $2$-connected by Theorem \ref{thm_BeltraLanteri}. Finally, $D$ is very ample by \cite[Lemma 4.4.1]{Matteini}. 
In fact, we can use Theorem~\ref{thm_Pisv} to generalize Example~\eqref{Example_Matteini} to the following:
\begin{corollary}[of Theorem~\ref{thm_Pisv}]
\label{cor_Matteini_generalized}
Let $(S,i)$ be a very general K3 surface with anti-symplectic involution of type $(10-d,10-d,1)$ with $3 \leq d \leq 8$, so that the quotient surface $T$ is a del Pezzo surface of degree $d$.
Let $C=-K_T$.
Then the associated relative Prym variety $\cP_D$ is an irreducible symplectic variety, of dimension $2d$.
\end{corollary}

In order to get ``genuinely new'' examples, we first show the following general result.
\begin{proposition} \label{prop_exdPn}
Let $(S,i)$ be a very general K3 surface with anti-symplectic involution of type $(10-d,10-d,1)$ with $1 \leq d \leq 8$, so that the quotient surface $T$ is a del Pezzo surface of degree $d$. Let $C$ be a very ample divisor on $T$ such that $C^2>4$ and $C.B>2$. Assume further that $C^2$ is even when $d=8$. Then for $n \geq 3$, the associated relative Prym variety $\cP_{nD}$ is an irreducible symplectic variety of dimension $n^2C^2+n\frac{C.B}{2}$.
\end{proposition}
\begin{proof}
We show that $|C|$, and so $|nC|$ for every integer $n$, satisfies the conditions in Theorem \ref{thm_isv} for every $1 \leq d \leq 8$.
Indeed, by assumption $C$ is very ample, $C.B >2$ and $C^2\neq 4$. Note also that the cases A1)--A3) listed in Theorem~\ref{thm_BeltraLanteri} cannot happen, and we can exclude A4) as well: $T$ is a $\mathbb{P}^1$-bundle when $d=8$, but a section plus some fibers would have odd square, while $C^2$ is even. We conclude that $nC$ is 2-connected for every $n>0$. Since $D$ is ample, by \cite[Theorem~2.2.7]{Huy_bookK3} $nD$ is very ample for $n \geq 3$. The result follows from Theorem \ref{thm_isv}.   
\end{proof}

Applying Proposition \ref{prop_exdPn} we obtain the following examples of arbitrarily high dimension, starting from dimension 8, with a linear system $|C|$ that has not been considered before.

\begin{proposition}\label{prop_example del Pezzo}

Let $(S,i)$ be a very general K3 surface with anti-symplectic involution of type $(10-d,10-d,1)$ with $1 \leq d \leq 8$, so that the quotient surface $T$ is a del Pezzo surface of degree $d$.
Consider the divisor
$$C=4H-2E_1-E_2- \dots -E_{9-d}.$$
Then for every integer $n$ the associated relative Prym variety $\cP_{nD}$ is an irreducible symplectic variety of dimension $n^2(4+d)+n(2+d)$. 
\end{proposition}
\begin{proof}
Note that $C$ satisfies the assumptions in Proposition \ref{prop_exdPn}. Indeed, we have $C^2=16-4-(9-d-1)=4+d > 4$, $C^2=12$ for $d=8$ and $C.B=24-4-2(8-d)=4+2d>4$. Using Theorem~\ref{thm_diRocco}, one checks that $C$ is very ample. This proves the statement for $n \geq 3$.

In fact, we can show that $D$ is very ample on $S$. According to Theorem~\ref{thm_Reider}, since $S$ is a K3 surface and $D=f^{-1}C$ is ample, it is enough to exclude that there exists a non-zero effective divisor $E$ with $D.E=1$ or with $D.E=2$ and $E^2=0$.

We can exclude the case $D.E=1$, $E^2=0$ as follows. The rank-2 lattice generated by $D$ and $E$ would be isometric to the rank-2 lattice generated by $F:=D-(4+d)E$ and $E$, which is a hyperbolic plane $U$. Since the K3 surface $S$ is general, the invariant lattice $L^+$ is the whole Picard group of $S$. But then $U$ would be contained in the invariant lattice, which has rank $10-d$ and discriminant group $(\mathbb{Z}/2\mathbb{Z})^{10-d}$, contradicting the fact that $U$ is unimodular.

We show that also the case $D.E=2$, $E^2=0$ is not possible. First note that 
$$\{f^*H,f^*E_1,\dots,f^*E_{9-d}\}$$ is a $\mathbb{Q}$-basis for $\text{NS}(S)$. 
We compute the intersection numbers of these elements:
\[f^*H.f^*E_i=2H.E_i=0,\;\; f^*E_i.f^*E_j=2E_i.E_j=-2\delta_{ij},\;\;(f^*H)^2=2H^2=2.\]
It follows that the lattice spanned by $\{f^*H,f^*E_1,\dots,f^*E_{9-d}\}$ has signature $(1, 9-d)$ and discriminant group $(\mathbb{Z}/2\mathbb{Z})^{10-d}$, just like the 
invariant lattice $L^+=\text{NS}(S)$ of $S$.
We conclude that $L^+$ is isometric to $\langle f^*H, f^*E_1,\dots,f^*E_{9-d} \rangle$. Thus $\{f^*H,f^*E_1,\dots,f^*E_{9-d}\}$ is a $\mathbb{Z}$-basis of $\NS(S)$.
We can thus write $E=af^*H+ \sum_{i=1}^{9-d} b_if^*E_i$ for some integers $a, b_i$. Then the conditions $D.E=2$, $E^2=0$ are equivalent to
$$\begin{cases}
4a+2b_1+\sum_{i=2}^{9-d} b_i=1, \\
a^2= \sum_{i=1}^{9-d} b_i^2. 
\end{cases}$$
It is possible to check using a computer (for instance using WolframAlpha) that there are no integral solutions to this system. 
Hence, Theorem~\ref{thm_Reider} tells us that $D$ is very ample on $S$, implying the statement for $\cP_D$ and $\cP_{2D}$ as well.
\end{proof}

The techniques of the proof of Proposition~\ref{prop_example del Pezzo} can also be used to show that in 
Example~\eqref{Ex_SS}, the hypotheses of Theorem~\ref{thm_Pisv} are satisfied when $n\geq 2$ or $d\geq 2$. 
Hence, we obtain the following strengthening of the result in \cite[Section 3.2]{SS} recalled in Example~\eqref{Ex_SS}:

\begin{proposition}\label{prop_ex_non-primitive linear system}
Let $(S,i)$ be a very general K3 surface with anti-symplectic involution of type $(10-d,10-d,1)$ with $1 \leq d \leq 9$, so that the quotient surface $T$ is a del Pezzo of degree $d$. Let $C=-2nK_T$ for any $n\geq 1$. 
Assume that either $d\geq 2$ or $n\geq 2$.
Then $\cP_D$ is an irreducible symplectic variety of dimension $2n(2n + 1)d$ that does not admit a symplectic resolution.
\end{proposition}

\bibliography{Bibliography}                    
\bibliographystyle{halpha-sortingadapted}

\end{document}